\documentclass{article}
\usepackage[T1]{fontenc}
\usepackage{titlesec}
\usepackage[utf8]{inputenc}
\usepackage{amsmath}
\usepackage{amssymb}
\usepackage{amsfonts}
\usepackage{amsthm}
\usepackage{geometry}
\usepackage{babel}
\usepackage{xcolor}
\newtheorem{lemma}{Lemma}
\newtheorem{theorem}{Theorem}
\newtheorem{proposition}{Proposition}

\newtheorem{assumption}{Assumption}

\numberwithin{equation}{section}
\usepackage{url}
\usepackage[hidelinks]{hyperref}

\title{arvix_November2}
\author{sdiaconu }
\date{November 2025}

\begin{document}

\begin{center}
    \textbf{Admissible Sequences for Talagrand's \(\gamma_2\)-functional}
\end{center}

\begin{center}
     Simona Diaconu\footnote{Courant Institute, New York University, simona.diaconu@nyu.edu}
\end{center}

\begin{abstract}
    Suprema of random processes appear naturally in a plethora of disciplines, and Talagrand's majorizing theorem yields a geometric interpretation for them: for a centered Gaussian random process \((X_t)_{t \in T},\) \(\mathbb{E}[\sup_{t \in T}{X_t}]\) is comparable to the \(\gamma_2\)-functional of \(T,\) a quantity that depends solely on the space \((T,d),\) where \(d\) denotes the pseudometric \(d(u,v)=\sqrt{\mathbb{E}[(X_u-X_v)^2]}.\) Despite the explicit definition of this functional, an infimum over admissible sequences, this tool tends to be used exclusively as a means to bound the expectation of the supremum of a random process by that of another. This work considers the \(\gamma_2\)-functional as a proxy for the quantity of interest by constructing admissible sequences that are close to being optimal, and aims to provide a promising avenue towards understanding expectations of suprema of Gaussian random processes.
\end{abstract}


\section{Introduction}

For a pseudometric space \((T,d),\) Talagrand's \(\gamma_2\)-functional is defined as
\begin{equation}\label{talagfunc}
    \gamma_2(T)=\inf_{\mathcal{A}}{\sup_{t \in T}{\sum_{h \geq 0}{2^{h/2} \cdot d(t,\mathcal{A}_h)}}},
\end{equation}
where the infimum is taken over admissible sequences\footnote{Tropp~\cite{troppe} assumes the sequence is nested, i.e., \(\mathcal{A}_h \subset \mathcal{A}_{h+1}.\) Both definitions yield comparable functionals by noting that \(\sum_{0 \leq h \leq H}{2^{2^h}} \leq 2^{2^{H+1}}.\)} for \(T,\) i.e., 
\[\mathcal{A}=(\mathcal{A}_h)_{h \geq 0} \subset 2^{T},\hspace{0.3cm} |\mathcal{A}_0|=1,\hspace{0.3cm} |\mathcal{A}_{h}| \leq 2^{2^h} \hspace{0.5cm} (h \in \mathbb{N}).\] 
\par
A remarkable feature of this quantity is the celebrated Talagrand's majorizing measure theorem, which states that for any centered Gaussian process \((X_t)_{t \in T}\) and the canonical pseudodistance on \(T\) given by 
\[d(t,s)=||X_t-X_s||_{L^2}=\sqrt{\mathbb{E}[(X_t-X_s)^2]},\]
the following holds\footnote{Measurability issues can arise as \(\sup_{t \in T}{X_t}\) is not a priori a random variable: by convention, \[\mathbb{E}[\sup_{t \in T}{X_t}]:=\sup_{S \subset T,|S|<\infty}{\mathbb{E}[\max_{t \in S}{X_t}]}.\]
See subchapter \(10.1.4\) in Tropp~\cite{troppe} for alternatives.},
\begin{equation}\label{majorizingth}
    \mathbb{E}[\sup_{t \in T}{X_t}]=\Theta(\gamma_2(T)).
\end{equation}
In other words, there exist universal constants \(c,C>0\) such that
\[c \cdot \gamma_2(T) \leq \mathbb{E}[\sup_{t \in T}{X_t}] \leq C \cdot \gamma_2(T).\]
This result has far-reaching consequences: suprema of random processes appear in sundry disciplines, and their expectation is oftentimes vital for understanding the suprema themselves due to their Lipschitz nature and a concentration of measure phenomenon (see, for instance, De et al~\cite{recentp}, a recent work containing sparsity results via Talagrand's majorizing theorem alongside their applications to agnostic learning and tolerant property testing).
\par
This paper is concerned with (and motivated by) random matrices, specifically symmetric Gaussian matrices whose upper-diagonal submatrices have independent entries: that is,
\begin{equation}\label{zmodel}
    Z=(b_{ij}g_{ij})_{1 \leq i,j \leq d} \in \mathbb{R}^{d \times d}
\end{equation}
with \(b_{ij} \geq 0, b_{ij}=b_{ji},g_{ij}=g_{ji},\) \((g_{ij})_{1 \leq i \leq j \leq d}\) i.i.d., and \(g_{11}\overset{d}{=} N(0,1).\) It is known that the operator norm of such random matrices is comparable in expectation to the maximal length of their columns (or equivalently rows). That is,
\begin{equation}\label{eq}
    \mathbb{E}[||Z||]=\Theta(\mathbb{E}[\max_{1 \leq i \leq d}{||Ze_i||}]),
\end{equation}
a result by Latala et al.~\cite{latala} from \(2018.\) The lower bound is immediate from 
\[||M||=\sup_{u \in \mathbb{S}^{d-1}}{||Mu||} \geq \max_{1 \leq i \leq d}{||Me_i||},\] 
an inequality that holds for all symmetric matrices \(M \in \mathbb{R}^{d \times d},\) whereas optimal upper bounds require careful handling of what is sometimes called the variance pattern of \(Z,\) \((b^2_{ij})_{1 \leq i,j \leq d}.\) In particular, the upper bound is far from trivial: as van Handel~\cite{vanhandel2} shows (see Theorem \(1.2\) therein), the right-hand side is comparable to a quantity that can be succinctly described in terms of the variance pattern, 
\begin{equation}\label{eq2}
    \mathbb{E}[\max_{1 \leq i \leq d}{||Ze_i||}]=\Theta(\max_{1 \leq i \leq d}{\sqrt{\sum_{1 \leq j \leq d}{b^2_{ij}}}}+\max_{1 \leq i \leq d}{(\sqrt{\log{(i+1)}}b_{\sigma(i)(1)})}),
\end{equation}
where \(b_{i(1)}=\max_{1 \leq j \leq d}{b_{ij}},\) and \(\sigma\) is a permutation of \(\{1,2,\hspace{0.05cm}...\hspace{0.05cm},d\}\) with \(b_{\sigma(1)(1)} \geq b_{\sigma(2)(1)} \geq ... \geq b_{\sigma(d)(1)}.\) However, most methods yield bounds some logarithmic factor away from the right size: for instance, a consequence of theorem \(1\) in Oliveira~\cite{oliveira} is that
\[\mathbb{E}[||Z||] \leq (\sqrt{2\log{(2d)}}+\frac{\sqrt{2\pi}}{2}) \cdot \max_{1 \leq i \leq d}{\sqrt{\sum_{1 \leq j \leq d}{b^2_{ij}}}}\]
(his result is more general than this as the model considered is \(Z=\sum_{1 \leq i \leq n}{\epsilon_i A_i}\) for deterministic symmetric matrices \((A_i)_{1 \leq i \leq n} \subset \mathbb{R}^{d \times d},\) \((\epsilon_i)_{1 \leq i \leq n}\) i.i.d. with \(\epsilon_1\overset{d}{=}N(0,1)\) or \(\mathbb{P}(\epsilon_1=1)=\mathbb{P}(\epsilon_1=-1)=\frac{1}{2}:\) the bounds obtained are polynomial in \(||\sum_{1 \leq i \leq n}{A_i^2}||\) and \(\log{d}\)), while theorem \(1.3\) in van Handel~\cite{vanhandel2} gives that
\[\mathbb{E}[||Z||]=O(\sqrt{\log{\log{d}}} \cdot \mathbb{E}[\max_{1 \leq i \leq d}{||Ze_i||}]).\]
\par
Regarding proof techniques, Oliveira~\cite{oliveira} uses a matrix extension of \(\mathbb{P}(|X| \geq a) \leq e^{-a^2} \cdot \mathbb{E}[X^2]\) for \(a>0,\) sometimes called Bernstein's trick, \(\mathbb{P}(||X|| \geq a) \leq e^{-a^2} \cdot \mathbb{E}[tr(e^{X^2})],\) whereas van Handel~\cite{vanhandel2} employs trace computations. The latter method plays a crucial role in the proof of (\ref{eq}) by Latala et al.~\cite{latala}: both results rely on block decompositions of the matrix of interest and careful combinatorics that allow fine bounds for the trace expectations of their powers. The dimensions of the blocks grow at the rate of the maximal sizes of sets forming admissible sequences, i.e., \(2^{2^h},h \in \mathbb{N}:\) this scale, not surprising in light of Talagrand’s majorizing measure theorem, is, up to an extent, another of its merits.
\par
This paper aims to investigate the \(\gamma_2\)-functional by finding admissible sequences whose corresponding contributions to the infimum underlying \(\gamma_2\) are small, i.e., close to optimal. The reason for this is twofold: there seem to be no works that exploit (\ref{majorizingth}) via definition (\ref{talagfunc}), and the thrust of the majorizing theorem appears to be not fully taken advantage of. The trace method has far reaching applications and its merits can be hardly overestimated: nonetheless, one of its shortcomings is that it generally does not yield information about eigenvectors. The majorizing theorem can yield delocalization properties of the leading eigenvectors since for Gaussian random matrices \(Z,\) a common choice for the random process \(X,\) yielding the operator norm \(||Z||,\) is \(X_t=t^TZt, t \in \mathbb{S}^{d-1},\) and thus subsets \(T\) of the unit sphere with \(\gamma_2(T)\) small relatively to \(\gamma_2(\mathbb{S}^{d-1})\) can render a form of localization (or lack thereof) for the leading eigenvector.
\par
When it comes to the \(\gamma_2\)-functional, a difficulty in bounding it lies within the square roots underpinning it: for 
\[X_t=t^TZt \hspace{1cm} (t \in \mathbb{R}^{d}),\] 
and \(Z\) given by (\ref{zmodel}), 
\begin{equation}\label{ddist}
    d(x,y)=\sqrt{\mathbb{E}[(X_x-X_y)^2]}=\sqrt{\sum_{1 \leq i \leq d}{b_{ii}^2(x_i^2-y^2_i)^2}+4\sum_{1 \leq i<j \leq d}{b_{ij}^2(x_ix_j-y_iy_j)^2}},
\end{equation}
and the easier-to-handle quantity
\[\sum_{h \geq 0}{2^{h} \cdot (d(t,\mathcal{A}_h))^2}\]
solely yields a lower bound for \(\gamma_2\) via the inequality \(\sqrt{\sum_{1 \leq i \leq r}{a_i}} \leq \sum_{1 \leq i \leq r}{\sqrt{a_i}}\) 
for \(r \in \mathbb{N},(a_i)_{1 \leq i \leq r} \subset [0,\infty),\)
\begin{equation}\label{ser}
    \inf_{\mathcal{A}}{\sup_{t \in T}{\sum_{h \geq 0}{2^{h} \cdot (d(t,\mathcal{A}_h))^2}}} \leq (\gamma_2(T))^2.
\end{equation}
Before proceeding with this series, a few remarks on \(||Z||\) and \(\sup_{t \in T}{X_t}\) are in order. Since solely an upper bound for \(\mathbb{E}[||Z||]\) is needed, this can be taken as \(O(\gamma_2(T)):\) theorem~\(8.5.3\) in Vershynin~\cite{vershynin} implies that
\[\mathbb{E}[\sup_{t \in T}{|X_t-X_{t_0}|}]=O(\gamma_2(T))\]
for any \(t_0 \in T,\) whereby 
\[\mathbb{E}[||Z||]=\mathbb{E}[\sup_{t \in T}{|X_t|}]=O(\gamma_2(T)+\mathbb{E}[|X_{t_0}|])\]
and letting \(t_0=(1,0,\hspace{0.05cm}...\hspace{0.05cm},0) \in \mathbb{R}^d\) produces the desired result as \(X_{t_0}\overset{d}{=}N(0,b^2_{11}).\)
\par
Return now to (\ref{ser}): the main contribution of this work can be summarized as follows.

\begin{theorem}\label{mainth}
There exists an admissible sequence \((\mathcal{A}_h)_{h \geq 0}\) for \(T=\mathbb{S}^{d-1}\) with
\begin{equation}\label{weakgamma}
    \sup_{t \in T}{\sum_{h \geq 0}{2^{h} \cdot (d(t,\mathcal{A}_h))^2}} \leq 121,931,520 \cdot C^2,
\end{equation}
where
\begin{equation}\label{cdef}
    C=\max{(\max_{1 \leq i \leq d}{\sqrt{\sum_{1 \leq j \leq d}{b^2_{ij}}}},\max_{1 \leq i \leq d}{(\sqrt{\log_2{(i+1)}}b_{\sigma(i)(1)})})}
\end{equation}
for \(\sigma\) as defined immediately below (\ref{eq2}).
\end{theorem}
A few remarks on Theorem~\ref{mainth} are in order (the description of \((\mathcal{A}_h)_{h \geq 0}\) requires additional notation, and this is completed in section~\ref{subsect2}). 
First, once diagonal matrices, i.e., \(b_{ij}=0\) \((i \ne j),\) are taken care of, it is perfectly natural to attempt to generalize the rationale employed for them. A somewhat subtle advantage in this special situation is that the signs of the entries of \(x,y \in \mathbb{R}^d\) are irrelevant\footnote{By a slight abuse of notation, \(d\) denotes sometimes a pseudometric, and other times a dimension. Context should make apparent which one is used.} as 
\[d(x,y)=\sqrt{\sum_{1 \leq i \leq d}{b_{ii}^2(x_i^2-y^2_i)^2}}:\] 
nevertheless, in the general case, these signs matter. This suggests that at most \(\log_2{|\mathcal{A}_h|} \leq 2^h\) entries can be chosen in the \(h^{th}\) step, hinting that only this many entries can be approximated well (instead of the natural guess, \(2^{2^h}\)), estimations necessary to ensure the summands of the series in (\ref{talagfunc}) decay fast. Second, the sets appearing in the direct sums underpinning \(\mathcal{A}_h\) can be described via a complete graph on \(\{1,2,\hspace{0.05cm}...\hspace{0.05cm},d\}\) with colored edges, the color of each edge \(ij\) being fully determined by the size of the corresponding variance \(b^2_{ij}.\) Lastly, (\ref{weakgamma}) is a first step towards building optimal admissible sequences, one that will hopefully lead to a better understanding of the \(\gamma_2\)-functional as well as a further exploitation of this extremely powerful instrument at the hour of comprehending suprema not only of Gaussian, but also of subgaussian processes insofar as the upper bound implicit in (\ref{majorizingth}) holds in such cases as well (see, for instance, the proof of theorem~\(8.5.3\) in Vershynin~\cite{vershynin}).
\par
The proof of (\ref{weakgamma}) relies on separating off-diagonal and diagonal blocks of the matrix \(Z,\) all these submatrices having dimensions 
comparable to elements of \((2^{2^h})_{h \in \mathbb{N}}.\) A crucial difference between these two contributions is that the former can be treated similarly to the purely diagonal case \(b_{ij}=0\) \((i \neq j):\) namely, the distances \(d(t,\mathcal{A}_h)\) are upper bounded by \(d(t,\hat{t}_h)\) for an explicit \(\hat{t}_h \in \mathcal{A}_h,\) whereas the latter seem unattainable via this approach. This can be explained by the off-diagonal blocks having essentially contributions of different orders, which in turn yield sufficiently fast decay to ensure a uniform bound for the series in (\ref{weakgamma}): these savings fade quickly for the diagonal blocks, and recursive approaches also seemed destined to fail, despite the fractal nature of these matrices. These difficulties are resolved by the probabilistic method: a clean uniform bound for this contribution arises from a two-step procedure, first reducing the analysis of the diagonal blocks to solely considering one at a time, and second defining random variables that capture the structure of the entry variances. Thus, interestingly enough, in spite of the admissible sequences being explicitly described via the structure of the variance matrix \(B=(b^2_{ij})_{1 \leq i,j \leq d},\) the distances \(d(t,\mathcal{A}_h)\) are bounded by using random points \(\hat{t}_h \in \mathcal{A}_h.\)
\par
The remainder of the paper is organized as follows. Section~\ref{sect1} considers the special case 
\begin{equation}\label{diagb}
    b_{ij}=0  \hspace{0.5cm} (1 \leq i,j \leq d, i \ne j).
\end{equation}
Its main result, Theorem~\ref{mainth3}, is a simpler version of Theorem~\ref{mainth2}: the former defines an admissible sequence for the case (\ref{diagb}) in a few lines, while the latter does this for the general case in several stages. It must be said that although the operator norm of interest is crystal-clear for diagonal matrices, 
\[||Z||=\max_{1 \leq i \leq d}{(b_{ii} \cdot |g_i|)}\]
for \((g_i)_{1 \leq i \leq d}\) i.i.d., \(g_1\overset{d}{=}N(0,1),\) the \(\gamma_2\)-functional for \(T=\mathbb{S}^{d-1}\) is more involved. Section~\ref{subsect2} deals with full generality by splitting the matrix into blocks, analyzing off-diagonal contributions, separating further the diagonal blocks, and treating each of the latter separately: since defining the admissible sequence in Theorem~\ref{mainth2} is divided into several steps, all proofs are postponed to section~\ref{sectp} and solely the statements of the results alongside needed notation are presented. It is worthwhile noticing that despite the separation aforesaid being immediate for matrices (i.e., for a symmetric matrix \(M \in \mathbb{R}^{(b_1+...+b_n) \times (b_1+...+b_n)}\) formed of \(n\) diagonal non-overlapping blocks of dimensions \((b_i \times b_i)_{1 \leq i \leq n},\) both its eigenvalues and eigenvectors can be easily described in terms of those of these \(n\) submatrices), this does not occur for the \(\gamma_2\)-functional because the pseudodistance \(d\) does not split into individual contributions of the underlying blocks (its square does but the admissible sequences cannot be merged directly). Section~\ref{sectp}, as previously mentioned, consists of the proofs of the propositions used in the construction of Theorem~\ref{mainth}. Section~\ref{subsect6} concludes with a discussion of what goes amiss when attempting to extend the analysis herein to the series underlying the \(\gamma_2\)-functional and possible remedies.

\section{A Special Case}\label{sect1}

Under the notation in (\ref{eq2}), it is assumed without loss of generality that
\[\sigma(i)=i \hspace{0.5cm} (1 \leq i \leq d)\] 
(the operator norm as well as the maxima are invariant under similarity transformations of \(Z\) given by permutation matrices: see the paragraph preceding (\ref{bdist}) for more details). The overall goal can be reformulated as
finding \((\mathcal{A}_h)_{h \geq 0},\) subsets of \(T=\mathbb{S}^{d-1}\) with \(|\mathcal{A}_0|=1,|\mathcal{A}_{h}| \leq 2^{2^h}\) \((h \in \mathbb{N}),\) and 
\begin{equation}\label{desired}
    \sqrt{\sup_{t \in T}{\sum_{h \geq 0}{2^{h} \cdot (d(t,\mathcal{A}_h))^2}}}=O(\max_{1 \leq i \leq d}{\sqrt{\sum_{1 \leq j \leq d}{b^2_{ij}}}}+\max_{1 \leq i \leq d}{(\sqrt{\log{(i+1)}}b_{i(1)})}).
\end{equation}
In what follows, 
\[l_x=2^{2^x} \hspace{0.2cm} (x \in \mathbb{R}), \hspace{0.3cm} S_n=\{l_{n-1},l_{n-1}+1,\hspace{0.05cm}...\hspace{0.05cm},l_n-1\}\hspace{0.5cm} (n \in \mathbb{N}),\]
and by convention, \(\log_2{0}=0.\) The numbers \((l_x)_{x \in \mathbb{R}}\) are vital in the bounds below, a property used repeatedly being \(l_{x+\log_2{y}}=l_x^y\) for \(y>0,\) and in the definition of \(\gamma_2\) since admissibility is, modulo the containment conditions, equivalent to \(|\mathcal{A}_h| \leq l_h.\) 
\par
It is also assumed that
\[b_{1j}=0 \hspace{0.5cm} (1 \leq j \leq d), \hspace{0.5cm} d=l_{d_0+1}-1  \hspace{0.2cm} (d_0 \in \mathbb{N}).\]
The former does not restrict generality as the matrix \(B\) can always be embedded into one with the first row and column zero with the maxima in (\ref{desired}) unchanged, while the natural embedding of \(\mathbb{R}^d\) into \(\mathbb{R}^{D}\) for \(D \geq d, \log_2{\log_2{(D+1)}} \in \mathbb{N},\)
\[(x_1,x_2,\hspace{0.05cm}...\hspace{0.05cm},x_d) \to (x_1,x_2,\hspace{0.05cm}...\hspace{0.05cm},x_d,0,0,\hspace{0.05cm}...\hspace{0.05cm},0),\] 
allows the latter once the sphere of interest is seen as a subset of its larger counterpart. It must be mentioned that the motivation for the second assumption is notational convenience: \(l_{n}-1=1+\sum_{1 \leq m \leq n}{|S_m|},\) and the sets \((S_n)_{n \in \mathbb{N}}\) can be replaced by \((S'_n)_{n \in \mathbb{N}}, S'_n=\{a_{n-1},a_{n-1}+1,\hspace{0.05cm}...\hspace{0.05cm},a_n-1\}\) with 
\[a_n \in \mathbb{Z}_{\geq 0}, \hspace{0.5cm} a_0=2, \hspace{0.5cm} a_{n}<a_{n+1}, \hspace{0.5cm} a_{n+1}-a_{n}=O(l_{n+1})\] 
at a multiplicative cost for the bounds below. 
\vspace{0.2cm}
\par
This section considers the special case
\begin{equation}\label{easy}
    b_{ij}=\begin{cases}
    0, \hspace{2.2cm} i=j=1,\\
    \frac{1}{\sqrt{\log_2{(i+1)}}}, \hspace{0.8cm} i=j>1,\\
    0, \hspace{2.2cm}  i \ne j.
\end{cases}
\end{equation}
There are two reasons for doing this: it illustrates the correct scales relevant to admissible sequences (encompassed by Lemma~\ref{lemma1}), and it is the only matrix for which the distances of interest are upper bounded directly, i.e., \(d(t,\mathcal{A}_h) \leq d(t,\hat{t}_h)\) for an explicit \(\hat{t}_h \in \mathcal{A}_h.\)
\par
In this situation,
\[d(u,v)=\sqrt{\sum_{2 \leq i \leq d}{\frac{(u_i^2-v_i^2)^2}{\log_2{(i+1)}}}}:\]
for simplicity, replace the denominators by powers of \(2:\)
\[d_2(u,v)=\sqrt{\sum_{0 \leq i \leq d_0}{\frac{\sum_{2^{2^i} \leq j< 2^{2^{i+1}}}{(u_j^2-v_j^2)^2}}{2^i}}}\]
(note that \(d(u,v) \leq d_2(u,v)\)). The main result of this section can be now stated.

\begin{theorem}\label{mainth3} 
Let \(d_0 \in \mathbb{N},d=l_{d_0+1}-1,\) and
\[\mathcal{B}_h=
\{(0,(\sqrt{\frac{x_j}{2^{h-i}}})_{j \in S_i, 1 \leq i \leq h_0},(0)_{j \in S_i,h_0<i \leq d_0+1}): (x_2,x_3,\hspace{0.05cm}...\hspace{0.05cm},x_{l_{h_0}-1}) \in T_{h_0,2^{h-h_0}}\}, 
\]
for \(h_0=\min{(h,d_0+1)},h \in \mathbb{N},\) where
\[T_{h,p}=\{(x_2,x_3,\hspace{0.05cm}...\hspace{0.05cm},x_{l_{h}-1}):x_2,x_3,\hspace{0.05cm}...\hspace{0.05cm},x_{l_{h}-1} \in \mathbb{Z}_{\geq 0},\sum_{1 \leq i \leq h}{\frac{\sum_{j \in S_i}{x_j}}{2^{h-i}}} \leq p\}.\]
\par
Then the sets
\[\mathcal{A}_h=\begin{cases}
    \{(1,0,0,\hspace{0.05cm}...\hspace{0.05cm},0)\} \subset \mathbb{R}^d, \hspace{5.3cm} 0 \leq h \leq 2,\\
    \{(y,x): (0,x) \in \mathcal{B}_{h-2},y=\sqrt{1-||x||^2}\}, \hspace{2.5cm} h\geq 3
\end{cases}\]
form an admissible sequence for \(T=\mathbb{S}^{d-1},\) i.e.,
\[(\mathcal{A}_h)_{h \geq 0} \subset 2^{T},\hspace{0.3cm} |\mathcal{A}_0|=1,\hspace{0.3cm} |\mathcal{A}_{h}| \leq 2^{2^h} \hspace{0.5cm} (h \in \mathbb{N}),\] 
and
\begin{equation}\label{supd}
    \sup_{t \in T}{\sum_{h \geq 0}{2^{h} \cdot (d_2(t,\mathcal{A}_h))^2}} \leq 27.
\end{equation}
\end{theorem}

Subsection~\ref{subsect02} consists of a lemma bounding \(|T_{h,p}|,\) and subsection~\ref{subsect03} employs inequalities derived from this result to justify Theorem~\ref{mainth3}.

\subsection{Sets of Small Size}\label{subsect02}

In this subsection, the sets \(T_{h,p}\) are shown to have small sizes. 

\begin{lemma}\label{lemma1}
    For \(h \in \mathbb{N},p \in \mathbb{Z}_{\geq 0},\)
    \[|T_{h,p}| \leq 10^h \cdot l_{h+\log_2{p}}.\]
\end{lemma}

In particular, Lemma~\ref{lemma1} yields that
\begin{equation}\label{bd22}
    |T_{h,1}| \leq 10^h \cdot l_{h} \leq l_{h+2}
\end{equation}
by induction on \(h \in \mathbb{N}:\) the base case holds as \(10 \cdot 4 \leq 256=2^{2^{3}},\) and for \(h \geq 2,\) 
\[10^h \cdot l_{h} \leq 10^{2h-2} \cdot l_{h}=(10^{h-1} \cdot l_{h-1})^2 \leq l_{(h-1)+2}^2=l_{h+2}.\]
The proof of this lemma is presented next.

\begin{proof}
Let \(n_{h,p}=|T_{h,p}|.\) Induction on \(n+m\) for \(n,m \in \mathbb{N}\) yields that
\begin{equation}\label{key}
    a_{n,m}=|\{(x_1,x_2,\hspace{0.05cm}...\hspace{0.05cm},x_{n}):x_1,x_2,\hspace{0.05cm}...\hspace{0.05cm},x_{n} \in \mathbb{Z}_{\geq 0},x_1+x_2+...+x_n=m\}|=\binom{m+n-1}{n-1}
\end{equation}
as \(a_{n+1,m}=a_{n,m}+a_{n+1,m-1}\) (either \(x_{n+1}=0\) or \(x_{n+1}-1 \geq 0\)). Hence by taking 
\[k=\sum_{j \in S_h}{x_j} \in \{0,1,\hspace{0.05cm}...\hspace{0.05cm},p\},\] 
the following recurrence ensues for \(h \geq 2,\)
\begin{equation}\label{rec1}
    n_{h,p}=\sum_{0 \leq k \leq p}{\binom{k+l_{h}-l_{h-1}-1}{l_h-l_{h-1}-1} \cdot n_{h-1,2p-2k}}.
\end{equation}
It is shown by induction on \(h \in \mathbb{N}\) that
\[n_{h,p} \leq 10^h \cdot l_{h+\log_2{p}}.\]
The base case \(h=1\) holds as
\[n_{1,p}=|\{(x_2,x_3) \in \mathbb{Z}_{\geq 0}^2: x_2+x_3 \leq p\}|=\sum_{0 \leq k \leq p}{(k+1)}=\binom{p+2}{2} \leq (1+1)^{p+2}=4 \cdot 2^{p}<10 \cdot 2^{2p}=10^1 \cdot l_{1+\log_2{p}},\]
and in light of the recurrence (\ref{rec1}), it suffices to justify for \(h \geq 2,\)
\begin{equation}\label{bd1}
    \sum_{0 \leq k \leq p}{\binom{k+l_{h}-l_{h-1}-1}{l_h-l_{h-1}-1} \cdot l_h^{-k}} \leq 10
\end{equation}
to complete the induction step. \(p=0\) is clear from \(n_{h,0}=1,l_x>1,\) while for \(0 \leq k<p,\)
\[\frac{l_{h-1+\log_2{(2p-2k)}}}{l_{h+\log_2{p}}}=\frac{l_{h+\log_2{(p-k)}}}{l_{h+\log_2{p}}}=l_h^{(p-k)-p}=l_h^{-k},\]
with the term for \(k=p\) allowing the numerator above to be replaced by \(1\) (\(n_{h-1,0}=1 \leq 10^{h-1} \cdot 1\)).
\par
Inequality (\ref{bd1}) follows from (\ref{pp1}) and (\ref{p2}) below:
\[\sum_{0 \leq k \leq l_h}{\binom{k+l_{h}-l_{h-1}-1}{l_h-l_{h-1}-1} \cdot l_h^{-k}} \leq \sum_{0 \leq k \leq l_h}{\frac{(k+l_h-l_{h-1}-1)^k}{k!} \cdot l_h^{-k}}<\]
\begin{equation}\label{pp1}
    <\sum_{0 \leq k \leq l_h}{\frac{(k+l_h)^k}{k!} \cdot l_h^{-k}}=\sum_{0 \leq k \leq l_h}{\frac{(1+\frac{k}{l_h})^k}{k!}} \leq \sum_{0 \leq k \leq l_h}{\frac{2^k}{k!}}<e^2<9,
\end{equation}
alongside
\[\sum_{k>l_h}{\binom{k+l_{h}-l_{h-1}-1}{l_h-l_{h-1}-1} \cdot l_h^{-k}} \leq \sum_{k>l_h}{\frac{(k+l_h-l_{h-1}-1)^{l_h-l_{h-1}-1}}{(l_h-l_{h-1}-1)!} \cdot l_h^{-k}}<\]
\begin{equation}\label{p2}
    <\sum_{k>l_h}{\frac{(2k)^{l_h-l_{h-1}-1}}{(l_h-l_{h-1}-1)!} \cdot l_h^{-k}}\leq \frac{1}{1-\frac{e}{16}} \cdot \frac{(2l_h)^{l_h-l_{h-1}-1}}{(l_h-l_{h-1}-1)!} \cdot l_h^{-l_h} \leq \frac{1}{1-\frac{e}{16}} \cdot \frac{2^{l_h-l_{h-1}-1}}{(l_h-l_{h-1}-1)!}<1
\end{equation}
using that
\[\frac{(2k+2)^{l_h-l_{h-1}-1} \cdot l_h^{-k-1}}{(2k)^{l_h-l_{h-1}-1} \cdot l_h^{-k}} \leq \frac{1}{l_h} \cdot (1+\frac{1}{k})^{l_h} \leq \frac{1}{l_h} \cdot e^{\frac{l_h}{k}} \leq \frac{e}{l_2}=\frac{e}{16}:\]
\(l_h-l_{h-1}-1 \geq l_2-l_1-1=16-4-1=11\) and \(n \to \frac{2^n}{n!}\) is nonincreasing in \(n \in \mathbb{N},\) together entailing the last product in (\ref{p2}) is at most
\[\frac{1}{1-\frac{e}{16}} \cdot \frac{2^{11}}{11!} \leq \frac{1}{1-\frac{e}{16}} \cdot \frac{2^{11}}{(11/e)^{11}} \leq \frac{1}{1-\frac{e}{16}} \cdot 2^{-11/2}<\frac{1}{1-\frac{3}{16}} \cdot \frac{1}{32}=\frac{1}{26}<1\]
by \(n! \geq (\frac{n}{e})^n,\) and \(\frac{2e}{11}<\frac{6}{11}<\frac{1}{2^{1/2}}.\)
\end{proof}

\subsection{An Admissible Sequence}\label{subsect03}

This subsection consists of the proof of Theorem~\ref{mainth3}.

\begin{proof}
Begin by showing that 
\begin{equation}\label{contain}
    \mathcal{A}_h \subset T:
\end{equation}
for \(h \leq 2,\) this is immediate, and the range \(h \geq 3\) amounts to showing \(\mathcal{B}_h \subset B(0,1)=\{x \in \mathbb{R}^d: ||x|| \leq 1\}\) for all \(h \in \mathbb{N}.\) The latter containment holds since for \(z \in \mathcal{B}_h,\) 
\[\sum_{1 \leq q \leq l_{d_0+1}-1}{z_q^2}=\sum_{1 \leq i \leq h_0}{\frac{\sum_{j \in S_i}{x_j}}{2^{h-i}}}=2^{-(h-h_0)} \sum_{1 \leq i \leq h_0}{\frac{\sum_{j \in S_i}{x_j}}{2^{h_0-i}}}\leq 2^{-(h-h_0)} \cdot 2^{h-h_0}=1,\]
a consequence of \(x=(x_j)_{j \in S_i,1 \leq i \leq h_0} \in T_{h_0,2^{h-h_0}}.\) This completes the proof of (\ref{contain}).
\par
Continue with the condition on the sizes of \(\mathcal{A}_h:\) this clearly holds for \(h \leq 2,\) while for \(h \geq 3,\)
\[|\mathcal{A}_h|=|\mathcal{B}_{h-2}| \leq l_h,\]
which follows from
\begin{equation}\label{sizeb}
    |\mathcal{B}_h|  \leq l_{h+2} \hspace{0.5cm} (h \in \mathbb{N}).
\end{equation}
To see this, note that Lemma~\ref{lemma1} and (\ref{bd22}) yield that for \( h \leq d_0+1,\)
\[|\mathcal{B}_h|=|T_{h,1}| \leq n_{h,1} \leq l_{h+2},\]
while for \(h \geq d_0+2,\)
\[|\mathcal{B}_h|=|T_{d_0+1,2^{h-d_0-1}}| \leq 10^{d_0+1} \cdot l_{(d_0+1)+(h-d_0-1)} \leq 10^{h-1} \cdot l_h \leq l_{h+2}.\]
This concludes the justification of the first claim of Theorem~\ref{mainth3}. 
\par
Consider now (\ref{supd}), and begin by fixing \(t \in T.\) Inasmuch as
\begin{equation}\label{shiftt}
    \sum_{h \geq 0}{2^{h} \cdot (d_2(t,\mathcal{A}_h))^2}= \sum_{h \leq 2}{2^{h} \cdot (d_2(t,\mathcal{A}_0))^2}+\sum_{h \geq 3}{2^{h} \cdot (d_2(t,\mathcal{B}_{h-2}))^2} \leq 7+4\sum_{h \geq 1}{2^{h} \cdot (d_2(t,\mathcal{B}_{h}))^2}
\end{equation}
from 
\[(d_2(t,\mathcal{A}_{0}))^2 \leq \sum_{i>1}{t_i^4} \leq ||t||^2=1,\]
it suffices to justify that
\[\sum_{h \geq 1}{2^{h} \cdot (d_2(t,\mathcal{B}_h))^2} \leq 5.\]
\par
This is accomplished by defining \(\hat{t}_h \in \mathcal{B}_h\) with
\begin{equation}\label{wseries}
    \sum_{h \geq 1}{2^{h} \cdot (d_2(t,\hat{t}_h))^2} \leq 5
\end{equation}
together with \(d_2(t,\mathcal{B}_{h}) \leq d_2(t,\hat{t}_{h}).\) Let
\begin{equation}\label{th}
    \hat{t}_h=(0,(\sqrt{\frac{x_{hj}}{2^{h-i}}})_{j \in S_i, 1 \leq i \leq h_0},(0)_{j \in S_i,h_0<i \leq d_0+1}) \in \mathcal{B}_h
\end{equation}
for \(x_{hj}=\lfloor 2^{h-i} \cdot t^2_{j} \rfloor,j \in S_i, 1 \leq i \leq h_0,\) possible because 
\[\sum_{1 \leq i \leq h_0}{\frac{\sum_{j \in S_i}{x_{hj}}}{2^{h_0-i}}} \leq \sum_{1 \leq i \leq h_0}{\frac{2^{h-i} \sum_{j \in S_i}{t_j^2}}{2^{h_0-i}}} \leq 2^{h-h_0} \cdot ||t||^2=2^{h-h_0}\]
entails \(x_h=(x_{hj})_{j \in S_i, 1 \leq i \leq h_0} \in T_{h_0,2^{h-h_0}}.\) Now (\ref{wseries}) ensues from
\[\sum_{h \geq 1}{2^h \cdot (d_2(t,\hat{t}_{h}))^2}=\sum_{1 \leq i \leq d_0+1}{(\sum_{j \in S_i}{t_j^4})\sum_{h \leq i-1}{2^{h-i}}}+\sum_{1 \leq i \leq d_0+1}{\sum_{h \geq i}{\sum_{j \in S_i}{2^{h-i} \cdot (t_{j}^2-\frac{\lfloor 2^{h-i} \cdot t^2_{j} \rfloor}{2^{h-i}})^2}}} \leq\]
\[\leq \sum_{1 \leq i \leq d_0+1}{\sum_{j \in S_i}{t_j^2}}+\sum_{1 \leq i \leq d_0+1}{\sum_{j \in S_i}{4t_j^2}} \leq 1+4=5\]
by using \(t_j^4 \leq t_j^2,\) and
\begin{equation}\label{xineq7}
    \sum_{k \geq 0}{2^{k} \cdot (x-\frac{\lfloor 2^{k} \cdot x \rfloor}{2^{k}})^2} \leq 4x
\end{equation}
for \(x \in [0,1].\) The latter inequality is clear when \(x=0,\) and for \(x \in (0,1],\) let \(s \in \mathbb{N}\) with \(x \in (2^{-s},2^{-s+1}],\) whereby
\[\sum_{k \geq 0}{2^{k} \cdot (x-\frac{\lfloor 2^{k} \cdot x \rfloor}{2^{k}})^2} \leq \sum_{0 \leq k \leq s-1}{2^{k} \cdot x^2}+\sum_{k \geq s}{2^{k} \cdot \frac{1}{2^{2k}}} \leq 2^s \cdot x^2+2^{-s+1} \leq 2x+2x=4x\]
since \(0 \leq x-\frac{\lfloor 2^{k} \cdot x \rfloor}{2^{k}}=\frac{\{2^{k} \cdot x\}}{2^k} \leq \min{(x,2^{-k})}\) for all \(k \geq 0.\)
\end{proof}

\section{Block Decompositions}\label{subsect2}

Dealing with the general case requires several stages: these are described next and their proofs are contained in section~\ref{sectp}. The section concludes with justifying Theorem~\ref{mainth2},
which amounts to putting together all the phases below.


\subsection*{Rearranging \(Z\)}

The first step is a careful rearrangement of the matrix \(B=(b^2_{ij})_{1 \leq i,j \leq d}\) underpinning the model (\ref{zmodel}).

\begin{assumption}\label{ass0}
For \(Z\) given by (\ref{zmodel}), suppose \(d=l_{d_0+1}-1\) for some \(d_0 \in \mathbb{N},\)
\begin{equation}\label{fastdecay}
    \max_{j_1 \in S_{i_1},j_2 \in S_{i_2}}{b^2_{j_1j_2}} \leq \begin{cases}
\frac{C^2}{|S_{i_1}|^2}, \hspace{2.4cm} i_1 \leq i_2-3, \\
\frac{2C^2}{2^{i_1}}, \hspace{2.6cm} i_1 \leq i_2,
\end{cases}
\end{equation}
for all \(1 \leq i_1,i_2 \leq d_0+1,\) \(b_{1j}=0\) for \(1 \leq j \leq d,\) and
\[\max_{1 \leq i \leq d}{\sum_{1 \leq j \leq d}{b^2_{ij}}} \leq C^2\]
for some \(C>0.\) 
\end{assumption}

Assumption~\ref{ass0} will facilitate bounding the contributions of what can be thought of as off-diagonal blocks of \(B,\) and a separation will permit dealing with the diagonal blocks, one at a time. It must be said that the latter are, up to an extent, fractal in nature: they are square matrices, whose dimensions grow at the rate of \(l_i=2^{2^i}\) and whose largest entry is at most of order \(\frac{1}{2^i},\) a shared property that will be vital in the coming arguments. A key consequence of the justification of Assumption~\ref{ass0} (in subsection~\ref{varrear}) is that the desired bound for the series of interest in Theorem~\ref{mainth} is \(O(C^2):\) the value \(C\) above can be taken as in (\ref{cdef}). 

\subsection*{Adjusting \(d\)}

Consider a slight modification for the pseudometric \(d.\) Since 
\(x_ix_j-y_iy_j=x_j(x_i-y_i)+y_i(x_j-y_j),\) and \((a+b)^2 \leq 2(a^2+b^2),\) it follows that
\[d(x,y) \leq 2 \cdot \tilde{d}(x,y)\]
for 
\[\Tilde{d}(x,y)=\sqrt{\sum_{1 \leq i,j \leq d}{b_{ij}^2x^2_j(x_i-y_i)^2}+\sum_{1 \leq i,j \leq d}{b_{ij}^2y^2_i(x_j-y_j)^2}}\]
from
\[\sqrt{\sum_{1 \leq i \leq d}{b_{ii}^2(x_i^2-y_i^2)^2}+4\sum_{1 \leq i<j \leq d}{b_{ij}^2(x_ix_j-y_iy_j)^2}} \leq \sqrt{2\sum_{1 \leq i,j \leq d}{b_{ij}^2(x_ix_j-y_iy_j)^2}} \leq 2 \cdot \tilde{d}(x,y).\]
\par
A further simplification is possible: similarly to the diagonal case covered in section~\ref{sect1}, the distances appearing in the series of interest (\ref{weakgamma}) will be upper bounded by defining points \(\hat{t}_h \in \mathcal{A}_h\) (the difference in this case being that the points \(\hat{t}_h\) are random, whereas in the previous section, they were deterministic: recall (\ref{th})). In all choices below, the selected points \(\hat{t}_h\) have their entries dominated by those of \(t,\) i.e., \(\hat{t}_{hi} \cdot t_i \geq 0, |\hat{t}_{hi}| \leq |t_{i}|\) for all \(1 \leq i \leq d,\) whereby \(\tilde{d}\) can be replaced by
\[\Tilde{d}_2(x,y)=\sqrt{\sum_{1 \leq i,j \leq d}{b_{ij}^2x^2_j(x_i-y_i)^2}}\]
as \(b_{ij}=b_{ji}\) and \(\Tilde{d}(t,\hat{t}_h) \leq 2^{1/2} \cdot \Tilde{d}_2(t,\hat{t}_h).\) 
\par
The last adjustment to \(\tilde{d}_2\) is as follows: let \(v:\mathbb{R}^d \times \{1,2,\hspace{0.05cm}...\hspace{0.05cm},d\} \to [0,\infty)\) be given by
\[v(t,j)=\sum_{1 \leq i \leq d}{b_{ij}^2t_i^2},\]
and note that 
\[\Tilde{d}_2(x,y)=\sqrt{\sum_{1 \leq j \leq d}{v(x,j) \cdot (x_j-y_j)^2}}.\]
A useful property of the function \(v\) is that for \(x \in \mathbb{R}^d,\)
\begin{equation}\label{keyv}
    \sum_{1 \leq j \leq d}{v(x,j)}=\sum_{1 \leq i,j \leq d}{b_{ij}^2x_j^2}=\sum_{1 \leq j \leq d}{x_j^2(\sum_{1 \leq i \leq d}{b_{ij}^2})} \leq ||x||^2 \cdot \max_{1 \leq j \leq d}{\sum_{1 \leq i \leq d}{b_{ij}^2}}.
\end{equation}
Furthermore, several such functions will be used below: this motivates the following generalizations. 
For \(M \in \mathbb{R}_{\geq 0}^{d \times d},\) let
\[\tilde{d}_{2,M}(x,y)=\sqrt{\sum_{1 \leq j \leq d}{v_{M}(x,j) \cdot (x_j-y_j)^2}}, \hspace{0.5cm} v_{M}(t,j)=\sum_{1 \leq i \leq d}{M_{ij}t_i^2}:\]
that is, \(\tilde{d}_{2,M},v_M\) are the natural analogs of \(\tilde{d}_{2},v,\) respectively, when \(B=(b_{ij}^2)_{1 \leq i,j \leq d}\) is replaced by \(M.\)

\subsection*{Decomposing \(B\)}

Under Assumption~\ref{ass0}, let
\begin{equation}\label{decc}
    B=(b_{ij}^2)_{1 \leq i,j \leq d}=(b_{j_1j_2}^2\chi_{j_1 \in S_{i_1},j_2 \in S_{i_2},|i_1-i_2| \geq 3})+(b_{j_1j_2}^2\chi_{j_1 \in S_{i_1},j_2 \in S_{i_2},|i_1-i_2| \leq 2}):=B_0+B_1.
\end{equation}
At a high level, the matrix \(B_0\) captures the off-diagonal blocks and their contributions to the pseudodistance of interest, whereas the matrix \(B_1\) encapsulates the interactions of the diagonal blocks. More concretely, the pairs \((j_1,j_2)\) satisfying \(j_1 \in S_{i_1},j_2 \in S_{i_2},|i_1-i_2| \leq 2\) can be interpreted as forming the diagonal blocks of \(B,\) these amounting to \(B_1,\) while the remainder, encompassed by \(B_0\) can be thought of as off-diagonal. The former positions are similar in nature to their diagonal counterparts \((j,j)\) in the sense that the indices forming them are close to being equal when seen at the scale of \(j \to \log_2{\log_2{j}}.\) This closeness forfeits the existence of a fast decaying bound for its corresponding entries (notice the stark difference between the two types of bounds in (\ref{fastdecay})). Put differently, for \(B_0,\) the decay of the summands in the series arises from that of its entries, while for \(B_1,\) a key feature of the matrices underpinning it is that their entries share the best a priori upper bound (contrary to \(B,\) for which \(b_{j_1j_2}^2 \leq \frac{C^2}{2^{\max{(i_1,i_2)}}}\) for \(j_1 \in S_{i_1},j_2 \in S_{i_2}\)).
\par
To make the above dichotomy rigorous, decompose \(B_1\) into \(14\) block-diagonal matrices:  let
\[P=\{(q,r):q \in \{-1,0,1\},r \in \{0,1\}\} \cup \{(q,r):q \in \{-2,2\},r \in \{0,1,2,3\}\},\]
and take
\begin{equation}\label{decb1}
    B_1=\sum_{(q,r) \in P}{B_{1,q,r}}
\end{equation}
by employing the sets
\[\{(i_1,i_1+q):2|i_1-r,i_1 \in \mathbb{Z}\} \hspace{0.1cm} (q \in \{-1,0,1\},r \in \{0,1\}), \hspace{0.2cm} \{(i_1,i_1+q):4|i_1-r,i_1 \in \mathbb{Z}\} \hspace{0.1cm} (q \in \{-2,2\},r \in \{0,1,2,3\})\]
a partition of \(\{(i_1,i_2) \in \mathbb{Z}^2:|i_1-i_2| \leq 2\},\) and the matrices that can be naturally associated with them,
\[B_{1,q,r}=(b_{j_1j_2}^2\chi_{j_1 \in S_{i_1},j_2 \in S_{i_2},i_2=i_1+q,2|i_1-r})\in \mathbb{R}^{d \times d} \hspace{0.5cm} (q \in \{-1,0,1\},r \in \{0,1\}),\]
\[B_{1,q,r}=(b_{j_1j_2}^2\chi_{j_1 \in S_{i_1},j_2 \in S_{i_2},i_2=i_1+q,4|i_1-r})\in \mathbb{R}^{d \times d} \hspace{0.5cm} (q \in \{-2,2\},r \in \{0,1,2,3\}).\]
These matrices have a block diagonal structure: \(B_{1,q,r}\) is formed of blocks of dimensions \(|S_i| \times |S_{i+q}|,\) while in the block corresponding to \(i,\) the sum of each row is at most \(C^2,\) and each entry is at most \(\frac{8C^2}{2^{i}}\) by Assumption~\ref{ass0}. Furthermore, for such distinct dimensions \(i,j,\) either \(|q| \leq 1\) and \(|i-j| \geq 2,\) a consequence of \(2|i-j=(i-r)-(j-r),\) or \(|q|=2\) and \(|i-j| \geq 4\) from \(4|i-j,\) entailing that these blocks can be enlarged to be diagonal with the two aforesaid properties preserved and the block structure unaffected (i.e., no overlap is created by this augmentation: see (\ref{decb122}) below). 
\par
Because symmetry renders the descriptions in the last step below more intuitive and less technical than the lack of it does (see the paragraph below Proposition~\ref{lastprop}: also, the equivalence \(j \in N_{r}(M,i) \Leftrightarrow i \in N_r(M,j)\) for the sets \(N_r(\cdot,\cdot)\) defined by (\ref{nr}) is key in its proof),
let the final decomposition of \(B\) be
\begin{equation}\label{decb12}
    B=B_0+B_1=B_0+\sum_{(q,r) \in P}{B_{1,q,r}}=B_0+\sum_{r \in \{0,1\}}{B_{1,0,r}}+\sum_{r \in \{0,1,2,3\}}{(B_{1,2,r}+B_{1,-2,r+2})},
\end{equation} 
with the addition being modulo \(4\) in the indices of the last sum. All \(7\) matrices are symmetric and the last \(6\) are block-diagonal: both claims are immediate for the first three, and for the remaining four, they are consequences of
\[B_{1,2,r}+B_{1,-2,r+2}=(b_{j_1j_2}^2\chi_{j_1 \in S_{i_1},j_2 \in S_{i_2},(i_1,i_2) \in \{(4k+r,4k+r+2),(4k+r+2,4k+r),k \in \mathbb{Z}\}})\]
as this yields \(B_{1,2,r}+B_{1,-2,r+2}\) is symmetric and \((\overline{S}_{k,r} \times \overline{S}_{k,r})_{k \geq 0},\) where \(\overline{S}_{k,r}=S_{4k+r} \cup S_{4k+r+1} \cup S_{4k+r+2},\) generate its diagonal blocks (subject to their elements being at most \(d_0+1\)). For simplicity, rewrite (\ref{decb12}) as
\begin{equation}\label{decb122}
    B=B_0+\sum_{r \in \{0,1\}}{B_{1,0,r}}+\sum_{r \in \{0,1,2,3\}}{(B_{1,2,r}+B_{1,-2,r+2})}:=B_0+\sum_{1 \leq r \leq 6}{B_{sym,r}}.
\end{equation}
The claim regarding the supremum of interest follows by analyzing each matrix above separately, and using the following patching result. 

\begin{proposition}\label{glueprop}
    Suppose \(M=M_1+...+M_n,\) \(M_1,M_2,\hspace{0.05cm}...\hspace{0.05cm},M_n \in \mathbb{R}_{\geq 0}^{d \times d},\) \((M_j)_{1i}=0\) for all \(1 \leq i \leq d,\) and assume that for all \(1 \leq j \leq n,\) \((\mathcal{A}^{(j)}_h)_{h \geq 0}\) is an admissible sequence for \(\{x \in \mathbb{R}^{d}:||x|| \leq 1\}\) such that for all \(t \in \mathbb{S}^{d-1},\) \(h \geq 0,\) there exists \(\hat{t}^{(j)}_{h} \in \mathcal{A}^{(j)}_h\) with \(\hat{t}^{(j)}_{hi} \cdot t_i \geq 0,|\hat{t}^{(j)}_{hi}| \leq |t_i|\) \((2 \leq i \leq d),\) and
    \[\sum_{h \geq 0}{2^{h}\cdot (\tilde{d}_{2,M_j}(t,\hat{t}^{(j)}_h)})^2 \leq K\]
    for some \(K>0.\) Then there exists an admissible sequence \((\mathcal{A}_h)_{h \geq 0}\) for \(\mathbb{S}^{d-1}\) such that for all \(t \in \mathbb{S}^{d-1},\)
    \begin{equation}\label{summ}
        \sum_{h \geq 0}{2^{h}\cdot (\tilde{d}_{2,M}(t,\mathcal{A}_h)})^2 \leq  (\lceil \log_2{n} \rceil+1+2^{\lceil \log_2{n} \rceil} \cdot n) \cdot K.
    \end{equation}
    
\end{proposition}

Since \(\tilde{d}_{2,B}=\tilde{d}_2,\) it remains to use Proposition~\ref{glueprop} for the decomposition given by (\ref{decb122}): a crucial ingredient for this is constructing admissible sequences \((\mathcal{A}^{(j)}_h)_{h \geq 0},\) \(0 \leq j \leq 6\) for the \(7\) matrices in the right-hand side with the corresponding series uniformly bounded in their dimensions. 

\subsection*{Off-Diagonal Contributions: \(B_0\)}

An admissible sequence with its corresponding supremum in Proposition~\ref{glueprop} for \(\tilde{d}_{2,B_0}\) is constructed next. Begin with its backbone: a sequence of subsets of \(\mathbb{R}^d\) close to being admissible for the unit ball \(\{x \in \mathbb{R}^d: ||x|| \leq 1\}.\)
\par
Let \(\mathcal{A}_{0,0}=\{(0,0,\hspace{0.05cm}...\hspace{0.05cm},0)\} \subset \mathbb{R}^d,\) and for \(h \in \mathbb{N},\) take
\begin{equation}\label{a00h}\tag{\(\mathcal{A}_{0,h}\)}
    \mathcal{A}_{0,h}=\cup_{z \in \mathcal{C}_h}{\mathcal{D}_h(z)},
\end{equation}
with 
\[\mathcal{C}_h=\{(\sqrt{\frac{x_j}{2^h}})_{2 \leq j \leq \min{(d,2^h)}}: x_j \in \mathbb{Z}_{\geq 0}, \sum_{2 \leq j \leq \min{(d,2^h)}}{x_j} \leq 2^h\} \subset \mathbb{R}^{\min{(d,2^h)}-1},\]
\[\mathcal{D}_h(z)=\{(0,s_j \cdot \sqrt{\frac{x_j}{2^h}})_{2 \leq j \leq d}: x_j \in \mathbb{Z}_{\geq 0}, s_j \in \{1,-1\}, x_j\chi_{j \not \in I_h(z)}=0, \sum_{2 \leq j \leq d}{x_i} \leq 2^h\} \subset \mathbb{R}^{d},\]
\[I_h(z)=\{i: 1 \leq i \leq d, \sum_{2 \leq j \leq \min{(d,2^h)}}{b^2_{ij}z_j^2} \geq \frac{C^2}{2^{h}}\} \subset \{1,2,\hspace{0.05cm}...\hspace{0.05cm},d\}.\]

A few remarks are in order. The comment regarding \((\mathcal{A}_{0,h})_{h \geq 0}\) being almost admissible refers to the fact that the sizes of these sets are of the right order of magnitude: Lemma~\ref{sizeas} entails
\[|\mathcal{A}_{0,h}| \leq l_h^{5}\]
and note that \(l_h^5<l_h^8=l_{h+3}\) (in general, \(l_{h+K_0}\) is a good proxy for \(l_h\) by simply shifting the elements in the sequence). Furthermore, in all stages below apart from the last, the sequences built have their elements subsets of unit balls, not of spheres: everything could be done solely with spheres but it would be more cumbersome (due to requiring all the norms to be exactly \(1\)) and the final constant in Theorem~\ref{mainth2} would be larger (a row vanishing in the matrices in Proposition~\ref{propsep} would guarantee the built sequences are subsets of the spheres: see the paragraph preceding the statement of Theorem~\ref{mainth2}).
\par
The following proposition yields that the sets above generate an admissible sequence that achieves a uniform upper bound for the series of interest.

\begin{proposition}\label{propoff}
    Let \(\tilde{\mathcal{A}}_{0,h}=\mathcal{A}_{0,\max{(h-3,0)}}.\) Then \((\tilde{\mathcal{A}}_{0,h})_{h \geq 0}\) is an admissible sequence for \(\{x \in \mathbb{R}^d: ||x|| \leq 1\},\) and for all \(t \in \mathbb{S}^{d-1},\) \(h \geq 0,\) there is \(\tilde{t}^{(0)}_{h} \in \tilde{\mathcal{A}}_{0,h}\) with \(\tilde{t}^{(0)}_{hi} \cdot t_i \geq 0,|\tilde{t}^{(0)}_{hi}| \leq |t_i|\) \((2 \leq i \leq d)\) such that
    \[\sum_{h \geq 0}{2^{h}\cdot (\tilde{d}_{2,B_0}(t,\tilde{t}^{(0)}_h)})^2 \leq 175C^2.\]
\end{proposition}

The condition on the entries of \(\tilde{t}^{(0)}_{h}\) will appear in other results below as well: it serves two purposes, guaranteeing that these points do not yield worse  entrywise approximations than the zero vector does (i.e., \((\tilde{t}^{(0)}_{hi}-t_i)^2 \leq t_i^2\)) and ensuring that local estimations can be lifted to the global condition that the resulting vector has unit length (i.e., if \(|\tilde{t}_{i}| \leq |t_i|\) for some positions \(i \in \{2,3,\hspace{0.05cm}...\hspace{0.05cm},d\},\) then \(\tilde{t}\) can be adjusted to have norm \(1\) by passing it through the function \(s\) defined by (\ref{sdef}), or alternatively, by setting its first entry accordingly while leaving the values already chosen untouched and zeroing out the rest of the positions among \(\{2,3,\hspace{0.05cm}...\hspace{0.05cm},d\}\)).

\subsection*{Separating Diagonal Blocks: \(B_{sym,r}\)}

As discussed immediately before (\ref{decb122}), the matrices \((B_{sym,r})_{1 \leq r \leq 6}\) are block diagonal, and their components can be taken to be square with dimensions among \((|\overline{S}_{k,r}|)_{k \geq 0,0 \leq r \leq 3}.\) The first step towards constructing their corresponding admissible sequences in Proposition~\ref{glueprop} is separating these smaller blocks.

\begin{proposition}\label{propsep}
    Suppose \(d=l_{d_0+1}-1\) for some \(d_0 \in \mathbb{N},\) while \(M \in \mathbb{R}_{\geq 0}^{d \times d}\) has \(M=M^T,\) \(M_{1j}=0\) \((1 \leq j \leq d),\)
    \begin{equation}\label{diagg}
        M_{j_1j_2}=0 \hspace{0.5cm} (j_1 \in \tilde{S}_{i_1},j_2 \in \tilde{S}_{i_2},i_1 \ne i_2),
    \end{equation}
    and
    \[\max_{j_1,j_2 \in \tilde{S}_{i}}{M_{j_1j_2}} \leq \frac{K}{2^i} \hspace{0.5cm} (1 \leq i \leq d_0+1)\]
    for some \(K>0,\) \(\tilde{S}_i=\{a_{i-1},a_{i-1}+1,\hspace{0.05cm}...\hspace{0.05cm},a_i-1\}\) \((1 \leq i \leq d_0+1),\) \(2=a_0 \leq a_1 \leq ...\leq a_{d_0+1}=d+1.\) 
    \par
    If \((\mathcal{A}^{(i)}_h)_{h \geq 0}\) are admissible sequences for \(\{x \in \mathbb{R}^{|\tilde{S}_i|}:||x|| \leq 1\}\) for \(1 \leq i \leq d_0+1,\) then there exists an admissible sequence \((\tilde{\mathcal{A}}_h)_{h \geq 0}\) for \(\{x \in \mathbb{R}^{d}:||x|| \leq 1\}\) with
    \begin{equation}\label{boundss}
        \sum_{h \geq 0}{2^h \cdot (d_M(t,\tilde{\mathcal{A}}_h))^2} \leq 73K+16\max_{1 \leq i \leq d_0+1}{\sum_{h \geq 1}{2^h \cdot (d_{M^{(i)}}(t|_i,\hat{t}^{(i)}_h))^2}},
    \end{equation}
    for all \(t \in \mathbb{S}^{d-1}\) and \(\hat{t}^{(i)}_h=(\hat{t}^{(i)}_{hj})_{j \in \tilde{S}_i} \in \mathcal{A}^{(i)}_h\) with \(\hat{t}^{(i)}_{hj} \cdot t_j \geq 0,|\hat{t}^{(i)}_{hj}| \leq |(t|_i)_j|,\) where\footnote{By convention, the sums in the maximum corresponding to indices \(i\) with \(\tilde{S}_i=\emptyset\) are taken to be \(0.\) Note this is compatible with the edge case in which all indices fall within this category as \(M=0\) in such situations.} for \(1 \leq i \leq d_0+1,\)
    \[M^{(i)}=(M_{j_1j_2})_{j_1,j_2 \in \tilde{S}_i} \in \mathbb{R}^{|\tilde{S}_i| \times |\tilde{S}_i|}, \hspace{0.5cm} t|_i:=\frac{1}{\sqrt{\sum_{j \in \tilde{S}_i}{t_j^2}}+\chi_{\sum_{j \in \tilde{S}_i}{t_j^2}=0}}(t_j)_{j \in \tilde{S}_i} \in \mathbb{R}^{|\tilde{S}_i|}.\] 
\end{proposition}

By virtue of Proposition~\ref{propsep}, analyzing each constituent block separately and patching them should be enough for the task at hand, the last two steps in the proof of Theorem~\ref{mainth2}. A note on dimensions is in order: condition (\ref{diagg}) is equivalent to the matrix \(M\) being block diagonal, and all that is relevant about its (square) submatrices is that their dimensions grow at most at the rate of a shift of \((l_n)_{n \in \mathbb{N}},\) i.e., \((l_{n+K_0})_{n \in \mathbb{N}}\) (despite two such sequences having terms that are not comparable due to \(l_{x+1}=l_x^2,\) this extension is possible due to the series of interest solely increasing by at most a constant factor when applying such a shift: this will be used repeatedly in section~\ref{sectp}).

\subsection*{Uniformly Bounded Matrices}

Since the diagonal (square) blocks constituting the matrices \((B_{sym,r})_{1 \leq r \leq 6}\) have different dimensions (elements of \((|\overline{S}_{k,r}|)_{k \geq 0,0 \leq r \leq 3}\)), additional notation is needed to treat all of them. 
\par
Let \(w \in \mathbb{N},K>0.\) For \(M \in \mathbb{R}_{\geq 0}^{n \times n},n \leq l_w,\) satisfying
\begin{equation}\label{stronger}
   M=M^T, \hspace{0.5cm} \max_{1 \leq i,j \leq n}{M_{ij}} \leq \frac{K}{2^{w}}, \hspace{0.5cm} \max_{1 \leq j \leq n}{\sum_{1 \leq i \leq n}{M_{ij}}} \leq K,
\end{equation} 
define a sequence \((\mathcal{A}_{2,h}(M))_{h \geq 0}\) as
\begin{itemize}
    \item \(\mathcal{A}_{2,0}(M)=\{(0,0,0,\hspace{0.05cm}...\hspace{0.05cm},0)\} \subset \{x \in \mathbb{R}^{n}:||x|| \leq 1\};\)
    \item \(\mathcal{A}_{2,h}(M) \subset \{x \in \mathbb{R}^{n}:||x|| \leq 1\}\) arbitrary, subject to \((0,0,0,\hspace{0.05cm}...\hspace{0.05cm},0) \in \mathcal{A}_{2,h}(M), |\mathcal{A}_{2,h}(M)| \leq l_h\) for \(1 \leq h \leq w;\)
    \item 
    \(\mathcal{A}_{2,h}(M)=\cup_{y \in \mathcal{S}_{h},(V_0,V_1,\hspace{0.05cm}...\hspace{0.05cm},V_{h-w}) \in \mathcal{P}_{h}(y)}{\mathcal{D}_h(\cup_{0 \leq j \leq h-w}{N_j(M,V_j)})}\) for \(h \geq w+1,\)
\end{itemize}
where \(M(h)=2^{\lfloor 2^h/h \rfloor},\)
\[\mathcal{S}_h=\{(\frac{x_j}{M(h)})_{0 \leq j \leq h-w}: x_j \in \mathbb{Z}_{\geq 0}, \sum_{0 \leq j \leq h-w}{x_j} \leq 5 \cdot M(h)\} \subset \mathbb{R}^{h-w+1},\]
\[\mathcal{P}_{h}(y)=\{(V_0,V_1,\hspace{0.05cm}...\hspace{0.05cm},V_{h-w}):V_0,V_1,\hspace{0.05cm}...\hspace{0.05cm},V_{h-w} \subset \{1,2,\hspace{0.05cm}...\hspace{0.05cm},n\}, |V_j| \leq 2^{h-w-j} \cdot y_j\},\]
\[\mathcal{D}_h(V)=\{(s_j \cdot \sqrt{\frac{x_j}{2^h}})_{1 \leq j \leq n}: x_j \in \mathbb{Z}_{\geq 0}, s_j \in \{1,-1\}, x_j\chi_{j \not \in V}=0, \sum_{1 \leq j \leq n}{x_j} \leq 2^h\} \subset \mathbb{R}^{n}\hspace{0.5cm} (V \subset \{1,2,\hspace{0.05cm}...\hspace{0.05cm},n\}),\]
\begin{equation}\label{nr}
    N_r(M,j)=\{i: 1 \leq i \leq n, M_{ij} \in (\frac{K}{2^{w+r+1}},\frac{K}{2^{w+r}}]\} \hspace{0.5cm} (r \in \mathbb{Z}_{\geq 0},1 \leq j \leq n),
\end{equation}
\[N_r(M,V)=\cup_{v \in V}{N_r(M,v)} \hspace{0.5cm} (V \subset \{1,2,\hspace{0.05cm}...\hspace{0.05cm},n\}).\] 
\par
A few comments are in order. The square blocks forming the matrices of interest, \((B_{sym,r})_{1 \leq r \leq 6},\) satisfy (\ref{stronger}) for some \(w \in \mathbb{N}\) with \(l_w\) comparable to their dimensions. A consequence of (\ref{stronger}) is that for \(1 \leq j \leq n,\) \((N_r(M,j))_{r \in \mathbb{Z}_{\geq 0}}\) is a partition of \(\{i: 1 \leq i \leq n,M_{ij}>0\},\) the set of positions on the \(j^{th}\) row of \(M\) that yield (nonzero) contributions to \(\tilde{d}_{2,M}.\) Last but not least, similarly to \((\mathcal{A}_{0,h})_{h \geq 0},\) the sequences \((\mathcal{A}_{2,h}(M))_{h \geq 0}\) can be adjusted to form admissible sequences: Lemma~\ref{admlem} gives
\[|\mathcal{A}_{2,h}(M)| \leq l_{h+4},\]
and the following proposition shows they achieve a supremum that is uniformly bounded in \(w.\) 

\begin{proposition}\label{lastprop}
    Let \(w \in \mathbb{N},K>0,\) and suppose \(M \in \mathbb{R}_{\geq 0}^{n \times n},n \leq l_w\) satisfies (\ref{stronger}). Then the sequence \(\tilde{\mathcal{A}}_{2,h}= \mathcal{A}_{2,\max{(h-4,0)}}(M),h \geq 0\) is admissible for \(\{x \in \mathbb{R}^{n}:||x|| \leq 1\}\) and 
    \[\sup_{t \in \mathbb{S}^{n-1}}{\sum_{h \geq 0}{2^h \cdot (\tilde{d}_{M,2}(t,\tilde{\mathcal{A}}_{2,h}))^2}} \leq 990K.\]
\end{proposition}

Return now to an earlier comment in the introduction on coloring a complete graph. Notice that each element of \(\mathcal{A}_{2,h}(M)\) for \(h \geq w+1\) has at most \(5 \cdot 2^h\) entries that are among 
\[\{0,\pm{\sqrt{\frac{1}{2^h}}},\pm{\sqrt{\frac{2}{2^h}}},\hspace{0.05cm}...\hspace{0.05cm},\pm{\sqrt{\frac{2^h}{2^h}}}\},\] subject to the resulting vector having length at most \(1,\) and all the rest 
are zero (the threshold is obtained via a union bound and the simple observation that \(|N_r(M,v)| \leq 2^{w+r}\) for all \(v,r\)). However, these \(5 \cdot 2^h\) positions are not arbitrary: already the number of such configurations can be considerably larger than \(l_h\) since for \(w \geq 4, h \leq 2^{w-2},\) there could be
\[\binom{l_{w}}{5 \cdot 2^h}>l_{w+h+1}\]
subsets of size \(5 \cdot 2^h\) with elements among those of \(\{1,2,\hspace{0.05cm}...\hspace{0.05cm},n\}\) from
\[\binom{l_{w}}{5 \cdot 2^h} \geq \frac{(l_{w}-5 \cdot 2^h)^{5 \cdot 2^h}}{(5 \cdot 2^h)!} \geq (\frac{l_{w}}{10 \cdot 2^h})^{5 \cdot 2^h} \geq l_{w-1}^{5 \cdot 2^h}>l_{w-1+(h+2)}\]
by using 
\[5 \cdot 2^h \leq 5 \cdot l_{w-2}=l_{w} \cdot \frac{5}{l_{w-2}^{3}}<\frac{l_w}{2}, \hspace{0.7cm} \frac{10 \cdot l_{w-2}}{l_{w-1}}=\frac{10}{l_{w-2}}<1.\]
In particular, shifting to adjust sizes as in Proposition~\ref{propoff} does not work here insomuch as the cost would be of order \(2^{w}.\) Instead the positions aforesaid can be described by using a complete graph on \(S_w\) whose edges are colored with elements of \(\mathbb{Z}_{\geq 0} \cup \{\infty\},\) namely,
\begin{center}
    \(ij\) is assigned color \(r,\) where \(M_{ij} \in (\frac{K}{2^{w+r+1}},\frac{K}{2^{w+r}}],\)
\end{center}
and\footnote{This graph is undirected. The matrix \(M\) being symmetric, i.e., \(M=M^T,\) guarantees this coloring is well-defined.} so the sets consisting of the aforesaid special entry indices are given by unions of monochromatic stars (i.e., \(N_r(M,v)\)) subject to overall using no more than \(O(2^h)\) vertices (this latter bound seems unavoidable due to the arbitrary signs of the entries that need to be accounted for as well). The trade-off between color and size is captured by the conditions underlying \(\mathcal{P}_h:\) the larger the value of \(r,\) the larger the size of stars associated with it could be, and so the fewer must be selected.

\subsection*{Diagonal Contributions: \(B_{sym,r}\)}

The patching previously mentioned is the final step towards completing the desired construction. For \(A,A_1,A_2,\hspace{0.05cm}...\hspace{0.05cm},A_n \subset \mathbb{R}^{d}, r \in \mathbb{R},\) let
\[A_1+A_2+...+A_n=\{a_1+a_2+...+a_n: a_1 \in A_1, a_2 \in A_2,\hspace{0.05cm}...\hspace{0.05cm},a_n \in A_n\}, \hspace{0.4cm} rA=\{r \cdot a, a \in A\}.\]
\par
The admissible sequences for \((B_{sym,r})_{1 \leq r \leq 6}\) rely on the sequences \((\mathcal{A}_{2,h}(M))_{h \geq 0}\) defined below (\ref{stronger}): for \(1 \leq i \leq d_0+1,\) notice that \(B^{(i)}=(b^2_{j_1j_2})_{j_1,j_2 \in S_i} \in \mathbb{R}^{|S_i| \times |S_i|}\) satisfies (\ref{stronger}) with \(w=i,K=2C^2:\) due to the matrices underpinning the diagonal block structure of \((B_{sym,r})_{1 \leq r \leq 6}\) having slightly larger dimensions (see the sets \(\overline{S}_{k,r}\) defined above (\ref{decb122})), adjustments must be made. Recall that
\[\overline{S}_{k,r}=\{4k+r,4k+r+1,4k+r+2\} \hspace{0.5cm} (k \in \mathbb{Z}),\]
and let
\[n_{k,r}=\sum_{u \in \overline{S}_{k,r}}{|S_u|}=l_{4k+r+2}-l_{4k+r-1}, \hspace{0.5cm} \rho(k,r)=\max{(\lfloor \frac{k-r-2}{4} \rfloor,0)} \hspace{0.5cm} (k \in \mathbb{Z}):\]
under this notation, the diagonal blocks forming \(B_{sym,r}\) have dimensions 
\((n_{k,r})_{\chi_{r \leq 0} \leq k \leq \rho(d_0+1,r)}\) because by construction, nonzero entries in these blocks have positions \(j_1j_2\) with 
\[j_1 \in S_{i_1},\hspace{0.3cm} j_2 \in S_{i_2}, \hspace{0.3cm} \{i_1,i_2\}=\{4k+r,4k+r+2\},\] 
whereby solely the indexes \(k\) with \(4k+r>0,4k+r+2 \leq d_0+1\) are relevant. 
\par
For \(1 \leq r \leq 6,\) let
\(\mathcal{A}_{1,0}(B_{sym,r})=\{(0,0,\hspace{0.05cm}...\hspace{0.05cm},0)\} \subset \mathbb{R}^d,\) and when \(h \in \mathbb{N},\)
\begin{equation}\label{a11h}\tag{\(\mathcal{A}_{1,h}\)}
    \mathcal{A}_{1,h}(B_{sym,r})=\cup_{y \in \mathcal{L}_h}{(\sqrt{y}_1\mathcal{A}^{(1)}_{h+\lfloor \log_2{y_1} \rfloor}+\sqrt{y}_2\mathcal{A}^{(2)}_{h+\lfloor \log_2{y_1} \rfloor}+...+\sqrt{y}_{h_0}\mathcal{A}^{(h_0)}_{h+\lfloor \log_2{y_{h_0}} \rfloor})}
\end{equation}
for \(h_0=\min{(\rho(h,r),\rho(d_0+1,r))},M(h)=2^{\lfloor 2^h/h \rfloor},\)
\[\mathcal{L}_h=\{(\frac{x_j}{M(h)})_{1 \leq j \leq h}: x_j \in \mathbb{Z}_{\geq 0}, x_1+...+x_h \leq M(h)\},\]
and
\begin{equation}\label{a11h1}
    \mathcal{A}^{(i)}_{k}=\mathcal{A}^{(i)}_{k}(B_{sym,r})=\{(x_j \cdot \chi_{j \in S_q, q \in \overline{S}_{i-1,r}})_{1 \leq j \leq d}: x \in \mathcal{A}^{(i)}_{2,\max{(k-4,0)}}\},
\end{equation}
for \(1 \leq i \leq \rho(d_0+1,k),k \in \mathbb{Z},\) where by an abuse of notation \(x_j=0\) when \(j\) does not belong to the set containing the labels of the entries of \(x\) (this makes no difference due to the multiplication by characteristic functions) with 
\[\mathcal{A}^{(i)}_{2,h}=\mathcal{A}_{2,h}(\overline{B}_{sym,r,i})\subset \mathbb{R}^{n_{i-1,r}} \hspace{0.5cm} (h \geq 0),\]
where \(\overline{B}_{sym,r,i}\) is the \(i^{th}\) diagonal block of \(B_{sym,r},\) 
\[\overline{B}_{sym,r,i} \in \mathbb{R}^{n_{i-1,r} \times n_{i-1,r}}, \hspace{0.5cm} (\overline{B}_{sym,r,i})_{j_1j_2}=(\overline{B}_{sym,r})_{j_1j_2}\]
by using the labeling \(\mathbb{R}^{n_{i-1,r}}=\{(x_j)_{j \in S_q, q \in \overline{S}_{i-1,r}}:x_j\in \mathbb{R}\}\) (by convention, the matrix is empty when \(i=1,r=0\)). Notice that the self-replicating character of the diagonal blocks forming \((B_{sym,r})_{1 \leq r \leq 6}\) is reflected in the almost recurrent definition of \(\mathcal{A}_{1,h}(\cdot).\) 

\subsection*{An Admissible Sequence}

Putting together the pieces above yields the main result of this work. Let
\[\mathcal{A}_h=\overline{\mathcal{A}}_{\max{(h-3,0)}} \hspace{0.5cm} (h \geq 0),\]
where
\(\overline{\mathcal{A}}_{0}=\{(1,0,0,\hspace{0.05cm}...\hspace{0.05cm},0)\} \subset \mathbb{R}^{d},\)
\[\overline{\mathcal{A}}_{h}=s(m(\mathcal{A}^{(3,3)}_h+\sum_{1 \leq r \leq 6}{\mathcal{A}^{(sym,r)}_h}))=\{s(m(a)):a \in \mathcal{A}^{(3,3)}_h+\sum_{1 \leq r \leq 6}{\mathcal{A}^{(sym,r)}_h}\} \hspace{0.3cm} (h \in \mathbb{N}),\]
with 
\begin{equation}\label{sdef}
    s:\mathbb{R}^{d} \to \mathbb{S}^{d-1}, \hspace{0.5cm} s(x)=\begin{cases}
    (y,(x_j)_{2 \leq j \leq d}), y=\sqrt{1-\sum_{2 \leq j \leq d}{x_j^2}}, \hspace{1cm} ||x|| \leq 1,\\(1,0,\hspace{0.05cm}...\hspace{0.05cm},0), \hspace{4.9cm} ||x||>1,
\end{cases}
\end{equation}  
\begin{equation}\label{mdef}
    m(x_1,x_2,\hspace{0.05cm}...\hspace{0.05cm},x_n)=(0,max_{s}(x_{12},x_{22},\hspace{0.05cm}...\hspace{0.05cm},x_{n2}),max_{s}(x_{12},x_{22},\hspace{0.05cm}...\hspace{0.05cm},x_{n2}),\hspace{0.05cm}...\hspace{0.05cm},max_{s}(x_{1d},x_{2d},\hspace{0.05cm}...\hspace{0.05cm},x_{nd})),
\end{equation}
\[max_s(r_1,r_2,\hspace{0.05cm}...\hspace{0.05cm},r_{n})=\begin{cases}
    sgn(r_1) \cdot \max_{1 \leq i \leq n}{|r_i|}, \hspace{1.6cm} sgn(r_1)=sgn(r_2)=...=sgn(r_{n}),\\
    0, \hspace{4.9cm} else,
\end{cases}\]
for \(x_1,x_2,\hspace{0.05cm}...\hspace{0.05cm},x_n \in \mathbb{R}^d,\) \(x_{ij}=(x_i)_j\) \((1 \leq i \leq n, 1 \leq j \leq d),\) \(r_1,r_2,\hspace{0.05cm}...\hspace{0.05cm},r_{n} \in \mathbb{R},\) the sign function
\[sgn:\mathbb{R} \to \{-1,0,1\}, \hspace{0.5cm} sgn(r)=\begin{cases}
    -1, \hspace{0.5cm} r<0,\\
    0,\hspace{0.8cm} r=0,\\
    1,\hspace{0.8cm} r>0,
\end{cases}\]
and \(\mathcal{A}^{(3,3)}_h=\mathcal{A}_{0,\max{(h-3,0)}},\) \(\mathcal{A}^{(sym,r)}_h=\mathcal{A}_{1,\max{(h-2,0)}}(B_{sym,r})\) for \(1 \leq r \leq 6\) (recall (\ref{a00h}),(\ref{a11h})).
\par
The function \(s\) guarantees the sets \(\mathcal{A}_{h}\) are contained in \(\mathbb{S}^{d-1}:\) notice that \(s\) is not needed for the special case covered in section~\ref{sect1} due to the simple description of the admissible sequence therein, which explains why it has not been introduced earlier. This containment requirement is also the primary reason for assuming the first row (and column) of \(B\) vanish: this offers, up to a great extent, a degree of freedom, or in other words, it is a relaxation of the somewhat rigid requirement that all entries of \(t \in \mathbb{S}^{d-1}\) should be well-approximated as \(h \to \infty\) (a necessary condition if the series of interest is to be uniformly bounded in \(d \in \mathbb{N}\)). Although this assumption could be discarded (because it does not restrict generality), it is used throughout the arguments to approximate individual entries without worrying about norms. 
\begin{theorem}\label{mainth2}
The sequence \((\mathcal{A}_h)_{h \geq 0}\) is admissible for \(\mathbb{S}^{d-1},\) and
\begin{equation}\label{weakgamma2}
    \sup_{t \in \mathbb{S}^{d-1}}{\sum_{h \geq 0}{2^{h} \cdot (d(t,\mathcal{A}_h))^2}} \leq c_0 \cdot C^2
\end{equation}
for \(c_0=8 \cdot 60 \cdot (8 \cdot 73+16 \cdot 990 \cdot 16)=121,931,520.\)
\end{theorem}

\begin{proof}
    Fix \(t \in \mathbb{S}^{d-1}:\) the goal is showing that
    \[\sum_{h \geq 0}{2^{h} \cdot (d(t,\mathcal{A}_h))^2} \leq c\cdot C^2\]
    for \(c \geq c_0.\) Defining \(\tilde{t}_h \in \mathcal{A}_h\) with \(\tilde{t}_{hj} \cdot t_j \geq 0,|\tilde{t}_{hj}| \leq |t_j|\) for \(2 \leq j \leq d,\) and 
    \[\sum_{h \geq 0}{2^{h} \cdot (\tilde{d}_{2,B}(t,\tilde{t}_h))^2} \leq \frac{c}{8} \cdot C^2\]
    suffices since \(d(t,\mathcal{A}_h) \leq d(t,\tilde{t}_h) \leq 2 \cdot \tilde{d}(t,\tilde{t}_h) \leq 2^{3/2} \cdot \tilde{d}_{2,B}(t,\tilde{t}_h)\) (by virtue of the inequalities preceding (\ref{keyv})).
    Identity (\ref{decb122}) yields
    \[B=B_0+\sum_{1 \leq r \leq 6}{B_{sym,r}},\]
    whereby it is enough to prove that these matrices alongside their corresponding sequences in the definition of \((\mathcal{A}_h)_{h \geq 0}\) satisfy the two conditions in Proposition~\ref{glueprop} (admissibility for unit balls, and the approximations encompassed by \(\hat{t}^{(j)}_h\) for \(h \geq 0,0 \leq j \leq 6\)) for \(K=\frac{c}{8 \cdot 60} \cdot C^2\) due to
    \[\lceil \log_2{7} \rceil=3, \hspace{0.5cm} 3+1+2^3 \cdot 7=60:\]
    the proof of the proposition gives that the admissible sequence is \((\mathcal{A}_h)_{h \geq 0},\) and the conditions on the points \(\tilde{t}_h\) ensue as well (see subsection~\ref{patch}, in particular (\ref{thval})).
    \par
    Proposition~\ref{propoff} provides both conditions for \(B_0,\mathcal{A}_{h}^{(3,3)},\) and therefore \(\frac{c}{8 \cdot 60} \geq 175\) suffices, guaranteed by \(c \geq c_0.\) Take now \(1 \leq r \leq 6\) arbitrary, and consider \(B_{sym,r},\mathcal{A}^{(sym,r)}_{h}.\) Proposition~\ref{propsep} can be applied to 
    \[M=B_{sym,r},\hspace{0.4cm} K=8C^2, \hspace{0.4cm} 
    \mathcal{A}^{(i)}_h=\mathcal{A}_{2,\max{(h-4,0)}}^{(i)}(\overline{B}_{sym,r,i}) \hspace{0.3cm} (1 \leq i \leq \rho(r,d_0+1)),\]
    \[\tilde{S}_{j}=\begin{cases}
        S_1 \cup S_2 \cup ... \cup S_{r-1}, \hspace{2.8cm} j=1,\\
        S_{j-4+r} \cup S_{j-4+r+1} \cup S_{j-4+r+2}, \hspace{1cm} \frac{j}{4} \in \mathbb{N},\\
        S_{j-5+r+3}, \hspace{4.3cm} \frac{j-1}{4} \in \mathbb{N},j<d_0-r-4,\\
        S_{j-5+r+3} \cup ... \cup S_{d_0+1}, \hspace{2.3cm} \frac{j-1}{4} \in \mathbb{N},j \geq d_0-r-4,\\
        \emptyset, \hspace{5.6cm} else,
    \end{cases}\]
    for \(1 \leq j \leq d_{0}+1\) (the third and fourth branches ensure \(\cup_{1 \leq j \leq d_0+1}{\tilde{S}_j}=\{2,3,\hspace{0.05cm}...\hspace{0.05cm},d\}\)), and
    \[a_j=\begin{cases}
        2, \hspace{4.7cm} j=0,\\
        \max{(l_{r-1},2)},\hspace{3cm} 1 \leq j \leq 3,\\
        \max_{t \leq j, |\tilde{S}_t|>0}{(1+\max{\Tilde{S}_t})}, \hspace{0.9cm} 4 \leq j \leq d_0+1.
    \end{cases}\]
    To see this, note that
    \(M=M^T\) is block diagonal (by the construction of \(B_{sym,r}:\) see (\ref{decb122}) and the identity preceding it), \(M_{1j}=0\) for \(j \in \{1,2,\hspace{0.05cm}...\hspace{0.05cm},d\},\) the additional factor of \(8=2 \cdot 2^2\) is needed due to (\ref{fastdecay}) and the diagonal blocks described before (\ref{decb122}), while the sequences are admissible for \(\{x \in \mathbb{R}^{d}:||x||\leq 1\}\) by Proposition~\ref{lastprop} for 
    \[M=\overline{B}_{sym,r,i}, \hspace{0.3cm} w=(i-4+r+2) \cdot \chi_{4|i}, \hspace{0.3cm} K=16C^2\]
    (the case \(w=0\) is vacuous by construction) as the condition on the size follows from \(|\tilde{S}_i| \leq l_{i-4+r+2} \leq l_{i+1},\) and 
    \[B_{j_1j_2} \leq \frac{2C^2}{2^{i_1}} \leq \frac{16C^2}{2^{i_2+1}}\] 
    for \(j_1 \in S_{i_1},j_2 \in S_{i_2},i_2=i_1+2.\)
    The admissible sequence for \(\{x \in \mathbb{R}^{d}:||x||\leq 1\}\) built in Proposition~\ref{propsep} is 
    \[(\mathcal{A}^{(sym,r)}_{h})_{h \geq 0}:\] the initial sequence (\ref{inits}) is \(\mathcal{A}_{1,h}(B_{sym,r})\) because Lemma~\ref{admlem} yields that the sequences \((\mathcal{A}^{(i)}_{h}(B_{sym,r}))_{h \geq 0},\) given by (\ref{a11h1}), are admissible for \(\{x \in \mathbb{R}^d: ||x|| \leq 1\}.\) The desired points \(\hat{t}^{(sym,r)}_h\) are given by the proof of (\ref{boundss}) for the choice \(\hat{t}_h^{(i)} \in \mathcal{A}_h^{(i)}\) given behind the inequality in Proposition~\ref{lastprop} for \(t|_i\) (this vector has either norm \(1\) or \(0:\) in the latter case, choose \(\hat{t}_h^{(i)}=(0,0,\hspace{0.05cm}...\hspace{0.05cm},0) \in \mathbb{R}^{n_{i-1,r}}\)), and the desired inequalities are inherited ((\ref{deft}) gives that \(\hat{t}_h^{(i)} \in \mathcal{A}_h^{(i)}\) satisfy \(t_j \cdot (\hat{t}_h^{(i)})_j \geq 0,|(\hat{t}_h^{(i)})_j| \leq |t_j|\) for \(2 \leq j \leq d\)). Hence the conclusion of Proposition~\ref{propsep} yields
    \[\frac{c}{8 \cdot 60} \geq 8 \cdot 73+16 \cdot 990 \cdot 16\]
    is enough since the maximum in its bound (\ref{boundss}) is at most \(990 \cdot 16C^2\) by Proposition~\ref{lastprop}.
    \par
    This completes the proof of the theorem.
\end{proof}

\section{Proofs}\label{sectp} 

This section turns to the proofs of the results needed for Theorem~\ref{mainth2}: Assumption~\ref{ass0} (subsection~\ref{varrear}), Propositions~\ref{glueprop}-\ref{lastprop} (subsection~\ref{patch}-\ref{subsec23}),
and Lemmas~\ref{sizeas},\ref{admlem} (subsection~\ref{almost}).

\subsection{Rearranging Variances}\label{varrear}

This subsection justifies Assumption~\ref{ass0}, namely, it argues that it does not restrict generality.
\par
Similarly to section~\ref{sect1}, it can be supposed that 
\[b_{1j}=0 \hspace{0.5cm} (1 \leq j \leq d), \hspace{0.5cm} d=l_{d_0+1}-1  \hspace{0.2cm} (d_0 \in \mathbb{N}),\]
as well as
\[b^2_{j_1j_2} \leq \frac{C^2}{2^{\max{(i_1,i_2)}}}, \hspace{0.5cm} \max_{1 \leq i \leq d}{\sum_{1 \leq j \leq d}{b^2_{ij}}} \leq C^2\]
for all \(1 \leq i_1,i_2 \leq d_0+1, j_1 \in S_{i_1},j_2 \in S_{i_2},\) and some \(C \geq 0.\) This can be done by taking \(d_0 \in \mathbb{N}\) with \(l_{d_0}>d,\)
\[C=C(B)=\max{(\max_{1 \leq i \leq d}{\sqrt{\sum_{1 \leq j \leq d}{b^2_{ij}}}},\max_{1 \leq i \leq d}{(\sqrt{\log{(i+1)}}b_{i(1)})}}\]
for \(B=(b^2_{ij})_{1 \leq i,j \leq d} \in \mathbb{R}^{d \times d},\) padding \(B\) with zeros to convert it to \(\tilde{B} \in \mathbb{R}^{(l_{d_0}-1) \times (l_{d_0}-1)},\) i.e., \(\Tilde{B}_{ij}=b^2_{ij}\chi_{i,j \leq d},\)
and applying a similarity transformation to it. It is immediate that \(C(B)=C(\tilde{B}),\) and for the permutation \(\sigma:\{1,2,\hspace{0.05cm}...\hspace{0.05cm},l_{d_0}-1\} \to \{1,2,\hspace{0.05cm}...\hspace{0.05cm},l_{d_0}-1\}\) yielding the needed ordering, \(\tilde{B}_P=P^T\tilde{B}P,\) where \(P_{ij}=\chi_{i=\sigma(j)}\) has \((\tilde{B}_P)_{ij}=\sum_{k_1,k_2}{P_{k_1i}\tilde{B}_{k_1k_2}P_{k_2j}}=\tilde{B}_{\sigma(i)\sigma(j)},\) and \(d_{\tilde{B}_P}(x,y)=d_{\tilde{B}}(P^Tx,P^Ty)\) for
\begin{equation}\label{bdist}
    d_{B}(x,y)=\sqrt{\sum_{1 \leq i,j \leq d}{b_{ij}^2x^2_j(x_i-y_i)^2}}
\end{equation}
\par
Lastly, Lemma~\ref{lemma5}, stated and proved below, entails that a similarity transformation can be applied to \(B\) and as a result its entries satisfy the inequalities 
in Assumption~\ref{ass0}, while the condition on the sums of the rows is a direct consequence of their invariance under multiplications of \(B\) by permutation matrices, i.e.,
\[\{\sum_{1 \leq j \leq d}{b^2_{\tau(i)\tau(j)}},1 \leq i \leq d\}=\{\sum_{1 \leq j \leq d}{b^2_{\tau(i)j}},1 \leq i \leq d\}=\{\sum_{1 \leq j \leq d}{b^2_{ij}},1 \leq i \leq d\}\]
for any permutation \(\tau:\{1,2,\hspace{0.05cm}...\hspace{0.05cm},d\} \to \{1,2,\hspace{0.05cm}...\hspace{0.05cm},d\}.\)

\begin{lemma}\label{lemma5}
    Suppose \(d_0 \in \mathbb{N},\) 
    and let \(M \in \mathbb{R}^{(l_{d_0}-1)\times (l_{d_0}-1)}\) with \(M=M^T,\)
    \[0 \leq M_{j_1j_2} \leq \frac{K}{2^{\max{(i_1,i_2)}}}, \hspace{0.5cm} \max_{2 \leq i \leq l_{d_0}-1}{\sum_{2 \leq j \leq l_{d_0}-1}{M_{ij}}} \leq K\]
    for all \(1 \leq i_1,i_2 \leq d_0, j_1 \in S_{i_1},j_2 \in S_{i_2},\) and some \(K>0.\) 
    \par
    Then there exists a permutation matrix \(P \in \mathbb{R}^{(l_{d_0}-1)\times (l_{d_0}-1)}\) such that \(N=P^TMP\) satisfies
    \begin{equation}\label{decay}
        \max_{j_1 \in S_{i_1},j_2 \in S_{i_2}}{N_{j_1j_2}} \leq \begin{cases}
            \frac{K}{|S_{i_1}|^2}, \hspace{2.3cm} i_1 \leq i_2-3, \\
            \frac{2K}{2^{i_1}}, \hspace{2.6cm} i_1 \leq i_2,
        \end{cases}
    \end{equation}
    for all \(1 \leq i_1,i_2 \leq d_0.\)
\end{lemma}

\begin{proof}
    Begin with an inequality: for \(m \in \mathbb{N},\)
    \begin{equation}\label{mineq}
        \sum_{1 \leq i \leq m}{|S_{i}|^3}<|S_{m+2}|.
    \end{equation}
    To see this, use induction. For \(m=1,\) this holds from \(|S_1|^3=(2^2-1)^3=27<240=2^{2^3}-2^{2^2}=|S_3|.\) For the induction step, it suffices to show that for \(m \in \mathbb{N},\)
    \[|S_{m}|^3<|S_{m+2}|-|S_{m+1}|:\]
    this ensues from
    \[\frac{|S_{m+2}|-|S_{m+1}|}{|S_{m}|^3}=\frac{l_{m+2}-2l_{m+1}+l_m}{(l_m-l_{m-1})^3}>\frac{l_{m+2}-2l_{m+1}}{l^3_{m}}=l_m-\frac{2}{l_m}>l_m-1>1.\]
    \par
    The desired result follows by rearranging the columns of \(M\) alongside their counterpart rows (to preserve symmetry). Let \((t_n)_{n \in \mathbb{N}}\) be given by
    \[t_1=0, \hspace{0.5cm} t_{j}=|S_{i}|^2+t_{j-1} \hspace{0.5cm} (j \in S_i,i \in \mathbb{N}).\]
    For \(1 \leq i \leq d_0-2\) and \(j \in S_i,\) select \(|S_{i}|^2\) elements of \(\{t_j+1,t_j+2,\hspace{0.05cm}...\hspace{0.05cm},l_{d_0}-1\}\) corresponding to the largest entries among \((M_{js})_{s \geq t_j+1},\) and use them to construct a permutation matrix as described previously by rearranging the columns (and rows) to ensure these  \(|S_{i+1}|^2\) elements are on the positions 
    \[t_j+1,t_j+2,\hspace{0.05cm}...\hspace{0.05cm},t_j+|S_{i}|^2\] 
    on the \(j^{th}\) row. The inequalities for \(i_1 \leq i_2-3\) ensue from (\ref{mineq}) for \(m=i_1:\) the rearrangements described above for indices \(2,3,\hspace{0.05cm}...\hspace{0.05cm},l_{i_1}-1\) use solely the first \(|S_{i_1+2}|\) columns, and thus for any \(j_1 \in S_{i_1}\) and \(j_2>|S_{i_1+2}|,\) there are at least \(|S_{i_1}|^2\) entries on the \(j_1^{th}\) row, preceding the \(j_2^{th}\) one and at least as large as it, entailing the claim for \(j_1 \in S_{i_1},j_2 \in S_{i_2}\) since \(j_2 \in S_{i_2}\) gives \(j_2 \geq l_{i_2-1} \geq l_{i_1+2}>|S_{i_1+2}|\) (all entries are nonnegative, and the sum of each row remains at most \(K\) after any transformation as described above). Lastly, the case \(i_1 \leq i_2\) follows by noting that at each step, the desired property remains satisfied: at the beginning, this is a consequence of \(M_{j_1j_2} \leq \frac{K}{2^{\max(i_1,i_2)}},\) while at each subsequent step, this is a direct consequence of
     \[t_j \geq |S_{i-1}|^3>|S_{i}|\]
     for \(j \in S_i, i \geq 3\) from 
     \[\frac{|S_{i-1}|^3}{|S_{i}|}=\frac{(l_{i-1}-l_{i-2})^3}{l_i-l_{i-1}}=\frac{l^3_{i-2}(l_{i-2}-1)^3}{l_{i-2}^2(l_{i-2}^2-1)}=\frac{l_{i-2}(l_{i-2}-1)^2}{l_{i-2}+1}>\frac{4}{5} \cdot (4-1)^2>1,\] 
     with the cases \(i \leq 2\) being immediate (all the entries are at most \(\frac{K}{2}\)). This entails that for entries that are rearranged the minimum between their rows and columns can only increase, justifying the claim.
\end{proof}

\subsection{Patching Admissible Sequences}\label{patch}

This subsection consists of the justification of Proposition~\ref{glueprop}, a result that allows merging admissible sequences.

\begin{proof}
    Take \(\mathcal{A}_{0,0}=\{(1,0,\hspace{0.05cm}...\hspace{0.05cm},0)\} \subset \mathbb{S}^{d-1},\) and for \(h \in \mathbb{N},\)
    \[\mathcal{A}_{0,h}=s(m(\mathcal{A}^{(1)}_h+\mathcal{A}^{(2)}_h+...+\mathcal{A}^{(n)}_h))=\{s(m(a)):a \in \mathcal{A}^{(1)}_h+\mathcal{A}^{(2)}_h+...+\mathcal{A}^{(n)}_h\} \subset \mathbb{S}^{d-1},\]
    where \(s:\mathbb{R}^{d} \to \mathbb{S}^{d-1},\) \(m:(\mathbb{R}^d)^n \to \mathbb{R}^{d}\) are given by (\ref{sdef}) and (\ref{mdef}), respectively. Since \(((\mathcal{A}^{(j)}_h)_{h \geq 0})_{1 \leq j \leq n}\) are admissible, 
    \begin{equation}\label{inn89}
        |\mathcal{A}_{0,h}| \leq l_h^n=l_{h+\log_2{n}}.
    \end{equation}
    \par
    Take now
    \[\mathcal{A}_h=\mathcal{A}_{0,\max{(h-\lceil \log_2{n} \rceil,0)}}:\]
    this is an admissible sequence for \(\mathbb{S}^{d-1}\) from \(|\mathcal{A}_{0,0}|=|\mathcal{A}^{(1)}_{0}|=1,\) and (\ref{inn89}). It suffices to show (\ref{summ}) holds. To do so, fix \(t \in \mathbb{S}^{d-1},\) let \(\hat{t}_h \in \mathcal{A}_{0,0}=\mathcal{A}_h\) when \(h \leq \lceil \log_2{n} \rceil,\) and
    \[\hat{t}_h=s(m(\hat{t}^{(1)}_{h-\lceil \log_2{n} \rceil}+\hat{t}^{(2)}_{h-\lceil \log_2{n} \rceil}+...+\hat{t}^{(n)}_{h-\lceil \log_2{n} \rceil}))\in \mathcal{A}_h\]
    when \(h>\lceil \log_2{n} \rceil.\)
    The conditions on \((\hat{t}^{(j)}_h)_{j,h}\) entail that for \(h>\lceil \log_2{n} \rceil,\)
    \begin{equation}\label{thval}
        \hat{t}_{hi}=sgn(t_i) \cdot \max_{1 \leq j \leq n}{(|\hat{t}^{(j)}_{(h-\lceil \log_2{n} \rceil)i}|)}
    \end{equation}
    for \(2 \leq i \leq d,\) whereby
    \[\sum_{h>\lceil \log_2{n} \rceil}{2^{h}\cdot (\tilde{d}_{2,M}(t,\mathcal{A}_h)})^2 \leq \sum_{h>\lceil \log_2{n} \rceil}{2^{h}\cdot (\tilde{d}_{2,M}(t,\hat{t}_h)})^2 \leq \sum_{1 \leq j \leq n}{\sum_{h>\lceil \log_2{n} \rceil}{2^{h}\cdot (\tilde{d}_{2,M_j}(t,\hat{t}^{(j)}_{h-\lceil \log_2{n} \rceil})})^2},\]
    the latter inequality using \((M_j)_{1i}=0\) and \((\tilde{d}_{2,M})^2=\sum_{1 \leq j \leq n}(\tilde{d}_{2,M_j})^2.\)
    \par
    Therefore,
    \[\sum_{h \geq 0}{2^{h}\cdot (\tilde{d}_{2,M}(t,\mathcal{A}_h)})^2 =\sum_{h \leq \lceil \log_2{n} \rceil}{2^{h}\cdot (\tilde{d}_{2,M}(t,\mathcal{A}_h)})^2+\sum_{h>\lceil \log_2{n} \rceil}{2^{h}\cdot (\tilde{d}_{2,M}(t,\mathcal{A}_h)})^2 \leq\]
    \[\leq (\lceil \log_2{n} \rceil+1) \cdot K+2^{\lceil \log_2{n} \rceil} \cdot nK=(\lceil \log_2{n} \rceil+1+2^{\lceil \log_2{n} \rceil} \cdot n) \cdot K,\] 
    concluding the proof of (\ref{summ}).
\end{proof}

\subsection{Off-Diagonal Contributions}\label{subsec21}

This subsection justifies Proposition~\ref{propoff}, a result that treats off-diagonal entries.

\begin{proof}
    Fix \(t \in T:\) let \(\hat{t}^{(0)}_0 \in \mathcal{A}_{0,0},\) and \(\hat{t}^{(0)}_h \in \mathcal{A}_{0,h}\) be given by \(z_h=(\sqrt{\frac{\lfloor 2^h \cdot t_j^2 \rfloor}{2^{h}}})_{2 \leq j \leq \min{(d,2^h)}} \in \mathcal{C}_h,\) 
    \[\hat{t}^{(0)}_{hj}=\begin{cases}
    0, \hspace{4cm} j=1\\
    sgn(t_j) \cdot \sqrt{\frac{\lfloor 2^h \cdot t_j^2 \rfloor}{2^{h}}}, \hspace{1.5cm} j>1, j \in I_h(z_h),\\
     0, \hspace{4cm} j>1,j \not \in I_h(z_h), 
    \end{cases}\]
    for \(h \in \mathbb{N}.\) Similarly to (\ref{shiftt}), letting \(\tilde{t}^{(0)}_h=\hat{t}^{(0)}_{\max{(h-3,0)}}\) entails
    \[\sum_{h \geq 0}{2^{h}\cdot (\tilde{d}_{2,B_0}(t,\tilde{t}^{(0)}_h)})^2 \leq 15 \cdot (\tilde{d}_{2,B_0}(t,\hat{t}_0))^2+8\sum_{h \geq 1}{2^{h}\cdot (\tilde{d}_{2,B_0}(t,\hat{t}^{(0)}_h)})^2 \leq 15C^2+8\sum_{h \geq 1}{2^{h}\cdot (\tilde{d}_{2,B_0}(t,\hat{t}^{(0)}_h)})^2\]
    since \(\tilde{d}_{2,B_0}(x,y) \leq \tilde{d}_{2,B}(x,y),\) and 
    \begin{equation}\label{maxd}
        \tilde{d}_{2,B}(x,y) \leq \sqrt{\max_{1 \leq j \leq d}{v_B(x,j)}} \cdot ||x-y|| \leq C \cdot ||x|| \cdot ||x-y||
    \end{equation}
    via (\ref{keyv}). In particular, it suffices to show
    \begin{equation}\label{split1}
    \sum_{h \geq 1}{2^{h}\cdot (\tilde{d}_{2,B_0}(t,\hat{t}^{(0)}_h)})^2 \leq 20C^2
    \end{equation}
    insofar as the entry conditions are clearly satisfied, i.e., \(\tilde{t}^{(0)}_{hi} \cdot t_i \geq 0,|\tilde{t}^{(0)}_{hi}| \leq |t_i|\) for \(2 \leq i \leq d.\)
    \par
    Proceed now with (\ref{split1}), and recall that
\[(\tilde{d}_{2,B_0}(t,\hat{t}^{(0)}_h))^2=\sum_{1 \leq i \leq d_0+1}{\sum_{j \in S_i}{v_{B_0}(t,j) \cdot (t_j-\hat{t}^{(0)}_{hj})^2}}.\]
Decompose \(v_{B_0}(t,j)\) into four components, by rewriting it as a summation over \(j_1 \in S_{i_1},1 \leq i_1 \leq d_0+1,\) and splitting the range of \(i_1\) into 
\[i_1 \leq m(h),\hspace{0.5cm} m(h)<i_1 \leq i-3, \hspace{0.5cm} i-2 \leq i_1 \leq i+2, \hspace{0.5cm} i_1 \geq i+3,\]
where \(m(h)=\lfloor \log_2{h} \rfloor.\) Denote their contributions to \(v_{B_0}\) by \((v_s)_{1 \leq s \leq 4},\) and analyze them separately: it is shown below that their inputs to the left-hand side term in (\ref{split1}) are at most \(12C^2,4C^2,0,4C^2,\) respectively.
\par
\(1.\) For \(z_h(t)=(0,(\sqrt{\frac{\lfloor 2^h \cdot t_j^2\rfloor}{2^h}})_{j \in S_1 \cup S_2 \cup ... \cup S_{\min{(m(h),d_0+1)}}},(0)_{j \in S_i, \min{(m(h),d_0+1)}<i \leq d_0+1}) \in \mathbb{R}^{d},\) 
\begin{equation}\label{diffs}
    0 \leq v_1(t,j)-v_1(z_h(t),j)= \sum_{j_1 \in S_{i_1},i_1 \leq \min{(m(h),d_0+1)}}{b_{j_1j}^2 \cdot (t_{j_1}^2-\frac{\lfloor 2^h \cdot t_{j_1}^2\rfloor}{2^h})} \leq \frac{1}{2^h}\sum_{j_1 \in S_{i_1},i_1 \leq d_0+1}{b_{j_1j}^2} \leq \frac{C^2}{2^{h}}.
\end{equation}
Since \(u=v_1+v_2\) is independent of \(h,\) the terms corresponding to \(j\) with \(v_1(t,j)+v_2(t,j) \leq \frac{C^2}{2^{h-2}}\) can be taken care of via
\[\sum_{i,h}{2^h\sum_{j \in S_i,u(t,j) \leq \frac{C^2}{2^{h-2}}}{u(t,j) \cdot (t_j-\hat{t}^{(0)}_{hj})^2}} \leq \sum_{1 \leq i \leq d_0+1}{\sum_{j \in S_i}{u(t,j) \cdot t_j^2 \cdot \sum_{2^h \leq \frac{4C^2}{u(t,j)}}{2^{h}}}} \leq\]
\begin{equation}\label{smallv}
    \leq \sum_{1 \leq i \leq d_0+1}{\sum_{j \in S_i}{u(t,j) \cdot t_j^2 \cdot \frac{8C^2}{u(t,j)}}}=8C^2 \cdot ||t||^2=8C^2.
\end{equation}
Consider the remaining positions \(j\) with \(v_1(t,j)+v_2(t,j)>\frac{C^2}{2^{h-2}}.\) Inequality (\ref{fastdecay}) implies
\begin{equation}\label{v2}
    v_2(t,j) \leq \sum_{j_1 \in S_{i_1}, m(h)<i_1 \leq i-3}{b_{j_1j}^2\cdot t_{j_1}^2} \leq \max_{j_1 \in S_{i_1}: m(h)<i_1 \leq i-3}{b_{j_1j}^2} \leq \frac{C^2}{|S_{m(h)+1}|^2} \leq \frac{4C^2}{l_{m(h)+1}^2} \leq \frac{4C^2}{2^{2h}},
\end{equation}
whereby if \(v_1(t,j)+v_2(t,j)>\frac{C^2}{2^{h-2}},\) then (\ref{diffs}) and (\ref{v2}) render
\[v_1(z_h(t),j)=[v_1(t,j)+v_2(t,j)]-[v_1(t,j)-v_1(z_h(t),j)]-v_2(t,j)> \frac{C^2}{2^{h-2}}-\frac{C^2}{2^{h}}-\frac{4C^2}{2^{2h}} \geq \frac{C^2}{2^{h}}\] 
as \(4-1-\frac{4}{2^{h}} \geq 4-1-2=1.\)
Hence such entries have \(|\hat{t}^{(0)}_{hj}|=\sqrt{\frac{\lfloor 2^h \cdot t_j^2\rfloor}{2^h}}\) (i.e., \(j \in I_h(z_h)\)), making their contribution at most
\[\sum_{i,h}{2^h\sum_{j \in S_i}{v(t,j) \cdot (t_j-\hat{t}^{(0)}_{hj})^2}} \leq \sum_{1 \leq i \leq d_0+1}{\sum_{j \in S_i}{v(t,j)\sum_{h \geq 0}{2^h \cdot (|t_j|-\sqrt{\frac{\lfloor t_j^2 \cdot 2^{h} \rfloor}{2^{h}}})^2}}} \leq \sum_{1 \leq i \leq d_0+1}{4\sum_{j \in S_i}{v(t,j)}} \leq 4C^2\]
from
\begin{equation}\label{ineqxx}
    \sum_{h \geq 0}{2^h \cdot (\sqrt{x}-\sqrt{\frac{\lfloor 2^{h} \cdot x \rfloor}{2^{h}}})^2} \leq \sum_{h \geq 0}{\frac{2^h \cdot (x-\frac{\lfloor 2^{h} \cdot x \rfloor}{2^{h}})^2}{x}} \leq 4 \hspace{1cm} (x \in [0,1]),
\end{equation}
a consequence of \((\sqrt{a}-\sqrt{b})^2=\frac{(a-b)^2}{(\sqrt{a}+\sqrt{b})^2} \leq \frac{(a-b)^2}{a}\) \((a>0, b \geq 0),\) alongside (\ref{xineq7}), 
\[\sum_{k \geq 0}{2^{k} \cdot (x-\frac{\lfloor 2^{k} \cdot x \rfloor}{2^{k}})^2} \leq 4x \hspace{1cm} (x \in [0,1]).\]
\par
\(2.\)  Inequality (\ref{v2}) and \(t_j \cdot \hat{t}^{(0)}_{hj} \geq 0, |t_j| \leq |\hat{t}^{(0)}_{hj}|\) \((2 \leq j \leq d)\) imply
\[\sum_{i,h}{2^h\sum_{j \in S_i}{v_2(t,j) \cdot (t_j-\hat{t}^{(0)}_{hj})^2}} \leq \sum_{i,h}{2^h \cdot \frac{4C^2}{2^{2h}}
\sum_{j \in S_i}{t_j^2}} \leq 4C^2\sum_{h \geq 1}{2^{-h}}=4C^2.\]
\par
\(3.\) By construction, \((B_0)_{j_1j_2}=0\) for \(j_1 \in S_{i_1},j_2 \in S_{i_2},|i_1-i_2| \leq 2,\) whereby \(v_3=0.\) 
\par
\(4.\) Inequality (\ref{fastdecay}) yields
\[v_4(t,j)=\sum_{j_1 \in S_{i_1},i_1 \geq i+3}{b_{j_1j}^2 \cdot t_{j_1}^2} \leq \max_{j_1 \in S_{i_1},i_1 \geq i+3}{b_{j_1j}^2} \leq \frac{C^2}{|S_{i}|^2},\]
entailing 
\[\sum_{i,h}{2^h\sum_{j \in S_i}{v_4(t,j) \cdot (t_j-\hat{t}^{(0)}_{hj})^2}} \leq \sum_{i,h}{2^h\sum_{j \in S_i}{\frac{C^2}{|S_{i}|^2} \cdot t_j^2}}.\]
Since \(v_4=0\) for \(j \in S_i, i \leq m(h)+3,\) solely the values \(j \in S_i, i \geq m(h)+4\) are left, whereby their contribution is at most
\[\sum_{i,h}{2^h\sum_{j \in S_i}{\frac{C^2}{|S_{i}|^2} \cdot t_j^2}} \leq \sum_{i,h}{2^h \cdot \frac{C^2}{|S_{i}|^2}} \leq \sum_{1 \leq i \leq d_0+1}{\frac{4C^2}{l_{i+1}} \cdot 2l_{i-3}}\]
using that \(i \geq m(h)+4 \geq \log_2{h}+3\) implies that \(h \leq 2^{i-3}\) as well as \(\sum_{h \leq K}{2^h} \leq 2^{K+1}.\) This amounts to at most
\[\sum_{i \geq 1}{\frac{4C^2}{l_{i+1}} \cdot 2l_{i-3}} \leq \sum_{i \geq 1}{\frac{8C^2}{l_{i}}} \leq \sum_{i \geq 1}{\frac{8C^2}{2^{i+1}}}=4C^2.\]

\end{proof}

\subsection{Separating Diagonal Blocks}\label{subsec22}

This subsection contains the proof of Proposition~\ref{propsep}, a result that allows analyzing the diagonal blocks of the matrices \((B_{sym,r})_{1 \leq r \leq 6}\) one at a time.

\begin{proof}
Define \((\mathcal{A}_{h})_{h \geq 0}\) in the same spirit as (\ref{a11h}): \(\mathcal{A}_0=\{(0,0,\hspace{0.05cm}...\hspace{0.05cm},0)\} \subset \mathbb{R}^d,\) while when \(h \in \mathbb{N},\)
\begin{equation}\label{inits}
    \mathcal{A}_{h}=\cup_{y \in \mathcal{L}_h}{(\sqrt{y}_1\mathcal{A}^{(1)}_{h+\lfloor \log_2{y_1} \rfloor}+\sqrt{y}_2\mathcal{A}^{(2)}_{h+\lfloor \log_2{y_1} \rfloor}+...+\sqrt{y}_{h_0}\mathcal{A}^{(h_0)}_{h+\lfloor \log_2{y_{h_0}} \rfloor})},
\end{equation}
for \(h_0=\min{(h,d_0+1)},M(h)=2^{\lfloor 2^h/h \rfloor},\) and
\[\mathcal{L}_h=\{(\frac{x_j}{M(h)})_{1 \leq j \leq h}: x_j \in \mathbb{Z}_{\geq 0}, x_1+...+x_h \leq M(h)\}.\]
Because \((\mathcal{A}^{(i)}_h)_{h \geq 0}\) is admissible for \(\{x \in \mathbb{R}^{|\tilde{S}_i|}:||x||\leq 1\}\) \((1 \leq i \leq d_0+1)\), the proof of the second inequality in Lemma~\ref{sizeas} can be used verbatim (i.e., (\ref{f7}) and (\ref{f8}) hold in this case as well) to derive
\[|\mathcal{A}_{h}| \leq l_h^3,\]
whereby \(\tilde{\mathcal{A}}_h=\mathcal{A}_{\max{(h-2,0)}}\) is an admissible sequence for \(\{x \in \mathbb{R}^d:||x|| \leq 1\}\) (as \(a_0=2,a_{d_0+1}=d+1\)). It suffices to show (\ref{boundss}) holds to conclude the proof. Similarly to the argument in Proposition~\ref{propoff},
\[\sum_{h \geq 0}{2^h \cdot (\tilde{d}_{2,M}(t,\tilde{\mathcal{A}}_h))^2} \leq 7 \cdot (\tilde{d}_{2,M}(t,\mathcal{A}_0))^2+\sum_{h \geq 3}{2^h \cdot (\tilde{d}_{2,M}(t,\mathcal{A}_{h-2}))^2} \leq 7K+4\sum_{h \geq 1}{2^h \cdot (\tilde{d}_{2,M}(t,\mathcal{A}_h))^2},\]
whereby what is left to justify is
\begin{equation}\label{splitt}
    \sum_{h \geq 1}{2^h \cdot (\tilde{d}_{2,M}(t,\mathcal{A}_h))^2} \leq \frac{33K}{2}+4\max_{1 \leq i \leq d_0+1}{\sum_{h \geq 1}{2^h \cdot (\tilde{d}_{2,M^{(i)}}(t|_i,\hat{t}^{(i)}_h))^2}}
\end{equation}
\par
Fix \(t \in T,\) and let 
\[s_i=\sum_{j \in \tilde{S}_i}{t_j^2} \hspace{0.5cm} (1 \leq i \leq d_0+1), \hspace{0.5cm} y=(\frac{\lfloor M(h) \cdot s_q \rfloor}{M(h)})_{1 \leq q \leq h_0}.\] 
Then \(y \in \mathcal{L}_h\) since \(\sum_{1 \leq i \leq h_0}{s_i} \leq ||t||^2=1,\) whereby
\[\tilde{t}_h =(\sqrt{y_i} \cdot (\hat{t}^{(i)}_{h+\lfloor \log_2{y_i} \rfloor})_j \cdot \chi_{j \in \tilde{S}_i,i \leq h_0})_{1 \leq j \leq d} \in \mathcal{A}_h.\]
It suffices to show 
\[\sum_{h \geq 1}{2^h \cdot (\tilde{d}_{2,M}(t,\tilde{t}_h))^2} \leq \frac{33K}{2}+4\max_{1 \leq i \leq d_0+1}{\sum_{h \geq 1}{2^h \cdot (\tilde{d}_{2,M^{(i)}}(t|_i,\hat{t}^{(i)}_{h}))^2}}\]
to deduce (\ref{splitt}): this is a consequence of (\ref{p1}), (\ref{prevs}), and (\ref{leftt}) (alongside (\ref{s31}) and (\ref{s32})), justified below. Recall that
\[\sum_{h \geq 1}{2^h \cdot (\tilde{d}_{2,M}(t,\tilde{t}_h))^2} =\sum_{h \geq 1, 1 \leq i \leq d_0+1}{2^h\sum_{j \in \tilde{S}_i}{v_M(t,j) \cdot (t_j-\tilde{t}_{hj})^2}}.\]
\par
The contribution of the range \(i>h\) is 
\begin{equation}\label{p1}\tag{\(s1\)}
    \sum_{h \geq 1, i>h}{2^h\sum_{j \in \tilde{S}_i}{v_M(t,j)t_j^2}} \leq \sum_{1 \leq i \leq d_0+1}{2^i\sum_{j \in \tilde{S}_i}{v_M(t,j)t_j^2}} \leq \sum_{1 \leq i \leq d_0+1}{K\sum_{j \in \tilde{S}_i}{t_j^2}}=K
\end{equation}
since for \(j \in \tilde{S}_i,\) \(\hat{t}_{hj}=0\) as \(i>h \geq h_0,\) and
\[v_M(t,j)=\sum_{1 \leq r \leq d}{M_{jr} \cdot t_r^2} \leq \max_{r \in \tilde{S}_i}{M_{jr}} \leq \frac{K}{2^i}.\] 
Similarly, the positions \(j\) with \(t_j^2 \leq 2^{-h+2}\) contribute at most
\begin{equation}\label{prevs}\tag{\(s2\)}
    \sum_{h \geq 1, t_j^2 \leq 2^{-h+2}}{2^h\sum_{j \in \tilde{S}_i}{v_M(t,j)\cdot t_j^2}} \leq \sum_{j}{v_M(t,j) \cdot \frac{8}{t_j^2} \cdot t_j^2} \leq 8\sum_{j}{v_M(t,j)} \leq 8K
\end{equation}
by employing that for \(j \in \tilde{S}_i,\) \(\tilde{t}_{hj} \cdot t_j \geq 0,\) 
\[|\tilde{t}_{hj}| \leq \sqrt{y_i} \cdot |(\hat{t}^{(i)}_{h+\lfloor \log_2{y_i} \rfloor})_j| \leq \sqrt{s_i} \cdot |(\hat{t}^{(i)}_{h+\lfloor \log_2{y_i} \rfloor})_j| \leq \sqrt{s_i} \cdot |(t|_i)_{j}|=|t_j|\] 
from the given properties of \((\hat{t}_h^{(i)})_{h \geq 0}\) (the contribution of \(j=1\) vanishes since \(M_{1l}=0,\) \(1 \leq l \leq d\)).
\par
It remains to consider the pairs \((j,h)\) with \(j \in \tilde{S}_i,i \leq h_0, t_j^2> 2^{-h+2}\) (\(i \leq h\) is equivalent to \(i \leq h_0\)). Let 
\[j(i)=\min{\{q \in \mathbb{N}: \exists j \in \tilde{S}_i, t_j^2 \geq 2^{-q}\}} \hspace{0.5cm} (1 \leq i \leq d_0+1).\] 
The pairs \((j,h), j \in \tilde{S}_i, h \leq j(i)+1\) are included in (\ref{prevs}) because \(2^{-h+2} \geq 2^{-j(i)+1}>t_j^2\) for all \(j \in \tilde{S}_i,\) while for \(h \geq j(i)+2,\)
\begin{equation}\label{sy}
    0 \leq s_i-y_i \leq M(h)^{-1} \leq 2^{-h+2} \leq \frac{s_i}{2}
\end{equation}
from \(\lfloor 2^h/h \rfloor \geq 2^h/h-1 \geq h-2\) (proved below (\ref{f7})).
In these cases, 
\[2^h\sum_{j \in \tilde{S}_i}{v_M(t,j) \cdot (t_j-\tilde{t}_{hj})^2}=2^h \cdot s_iy_i\sum_{j \in \tilde{S}_i}{v_{M^{(i)}}(t|_i,j) \cdot (\sqrt{\frac{s_i}{y_i}}\cdot (t|_i)_j-(\hat{t}^{(i)}_{h+\lfloor \log_2{y_i} \rfloor})_j)^2} \leq\]
\begin{equation}\label{leftt}\tag{\(s3\)}
    \leq 2^h \cdot 2s_iy_i\sum_{j \in \tilde{S}_i}{v_{M^{(i)}}(t|_i,j) \cdot (\sqrt{\frac{s_i}{y_i}}-1)^2\cdot \frac{t^2_j}{s_i}}+2^h \cdot 2s_iy_i\sum_{j \in \tilde{S}_i}{v_{M^{(i)}}(t|_i,j) \cdot ((t|_i)_j-(\hat{t}^{(i)}_{h+\lfloor \log_2{y_i} \rfloor})_j)^2}
\end{equation}
via \(v_M(t,j)=v_{M^{(i)}}((t_l)_{l \in \tilde{S}_i},j)\) when \(j \in \tilde{S}_i,\) and \((a+b)^2 \leq 2(a^2+b^2)\) \((a,b \in \mathbb{R}).\)
\par
The first sums in (\ref{leftt}) amount to at most
\[\sum_{h,i: h \geq j(i)+2}{2^h \cdot 2s_iy_i\sum_{j \in \tilde{S}_i}{v_{M^{(i)}}(t|_i,j) \cdot \frac{M(h)^{-1}}{s_i} \cdot \frac{t^2_j}{s_i}}} \leq \frac{K}{2}\sum_{h \geq 1}{\frac{2^{h+1}}{M(h)}} \leq\]
\begin{equation}\label{s31}\tag{\(s3.1\)}
    \leq K\sum_{1 \leq h \leq 3}{2^{h-\lfloor 2^h/h \rfloor}}+K\sum_{h \geq 4}{2^{h+1-2^h/h}} \leq \frac{7K}{2}+2K \cdot 2=\frac{15K}{2}
\end{equation}
using 
\[(\sqrt{\frac{s_i}{y_i}}-1)^2=\frac{(\sqrt{s_i}-\sqrt{y_i})^2}{y_i} \leq \frac{(s_i-y_i)^2}{y_i \cdot s_i} \leq \frac{s_i-y_i}{s_i} \leq \frac{M(h)^{-1}}{s_i}\] 
from (\ref{sy}), 
as well as \(\sup_{||x|| \leq 1}{v_{M^{(i)}}(x,j)}=\max_{j_1,j_2 \in \tilde{S}_i}{M_{j_1j_2}} \leq \frac{K}{2^i} \leq \frac{K}{2},\)
\[[h+2-2^{h+1}/(h+1)]-[h+1-2^h/h]=1-\frac{2^h(h-1)}{h(h+1)} \leq -1\] 
from \(2^{h-1}(h-1) \geq h(h+1)\) for \(h \geq 4\) by induction on \(h,\) and 
\[\frac{2^{h}h \cdot h(h+1)}{(h+1)(h+2) \cdot 2^{h-1}(h-1)}=\frac{2h^2}{(h-1)(h+2)}=1+\frac{h^2-h+2}{(h-1)(h+2)}>1.\]
\par
The second sums in (\ref{leftt}) contribute at most
\[\sum_{h,i: h \geq j(i)+2}{2^h \cdot 2s_iy_i\sum_{j \in \tilde{S}_i}{v_{M^{(i)}}(t|_i,j) \cdot ((t|_i)_j-(\hat{t}^{(i)}_{h+\lfloor \log_2{y_i} \rfloor})_j)^2}} \leq\]
\begin{equation}\label{s32}\tag{\(s3.2\)}
    \leq 4\max_{1 \leq i \leq d_0+1}{\sum_{h \geq j(i)+2}{2^{h+\lfloor \log_2{y_i} \rfloor} \cdot \sum_{j \in \tilde{S}_i}{v_{M^{(i)}}(t|_i,j) \cdot ((t|_i)_j-(\hat{t}^{(i)}_{h+\lfloor \log_2{y_i} \rfloor})_j)^2}}}
\end{equation}
because \(y_i \leq 2^{1+\lfloor \log_2{y_i} \rfloor}\) (this holds also for \(y_i=0\) under the convention \(\log_2{0}=0\)), and \(\sum_{1 \leq i \leq d_0+1}{s_i}=1.\) Lastly, the desired claim follows from changing the index of summation for the \(i^{th}\) sum given by
\[h \to h+\lfloor \log_2{y_i} \rfloor=h+\lfloor \log_2{(\frac{\lfloor M(h) \cdot s_i \rfloor}{M(h)})} \rfloor.\]
This can be done because the above function is increasing: \(i \to \frac{\lfloor M(h) \cdot s_i \rfloor}{M(h)}\) is nondecreasing as \(M(h+1) \geq M(h),\) a consequence of \(2^{h+1}/(h+1) \geq 2^h/h,\) which implies 
\[\frac{\lfloor M(h+1) \cdot s_i \rfloor}{M(h+1)}-\frac{\lfloor M(h) \cdot s_i \rfloor}{M(h)}=\frac{1}{M(h)} \cdot (\frac{\lfloor n(h) \cdot x(h,i) \rfloor}{n(h)}-\lfloor x(h,i) \rfloor) \geq 0\] for \(n(h)=\frac{M(h+1)}{M(h)} \in \mathbb{N},x(h,i)=M(h) \cdot s_i,\) and when \(h \geq j(i)+2,\) \(y_i \geq \frac{s_i}{2} \geq 2^{-j(i)-1} \geq 2^{-h+1}\) gives \(h+\lfloor \log_2{y_i} \rfloor \geq 1.\)
\end{proof}

\subsection{Uniformly Bounded Matrices}\label{subsec23}

This subsection contains the proof of Proposition~\ref{lastprop}, the ultimate component for Theorem~\ref{mainth2}: this yields random points that upper bound the distances of interest.

\begin{proof}
Begin by supposing 
\[M_{ij} \in \{0,\frac{K}{2^w},\frac{K}{2^{w+1}},\hspace{0.05cm}...\hspace{0.05cm},\frac{K}{2^{w+q}},\hspace{0.05cm}...\} \hspace{0.5cm} (i,j \in \{1,2,\hspace{0.05cm}...\hspace{0.05cm},n\}):\]
since for all \(i,j \in \{1,2,\hspace{0.05cm}...\hspace{0.05cm},n\},\) either \(M_{ij}=0,\) or \(M_{ij} \in (\frac{K}{2^{w+r+1}},\frac{K}{2^{w+r}}]\) for some \(r \in \mathbb{Z}_{\geq 0},\) this can be done at the cost of multiplying \(K\) by \(2,\) whereby showing the supremum of interest is at most \(495K\) suffices.
\par
Fix \(t \in T:\) for \(r \geq 0,\) take
\[v_r(t,j)=\sum_{ij \in E_r}{M_{ij}t_i^2},\hspace{0.5cm} \sigma_r=\sigma_r(t)=\sum_{ij \in E_r}{t_i^2t_j^2},\]
where \(E_r=\{ij: i \in N_r(M,j)\}=\{ij: j \in N_r(M,i)\}\) (here \(ij\) denotes an undirected edge between \(i\) and \(j\) when \(i,j\) are seen as vertices in an undirected graph: in particular, \(ij=ji\)), and note that
\[\sum_{r \geq 0}{\sigma_r}=||t||^4=1.\]
For \(r \geq 0\) with \(\sigma_r>0,\) let \(X^{(r)}\) be the discrete random variable with probability mass distribution 
\[\mathbb{P}(X^{(r)}=i)=\frac{\sum_{j \in E_r}{(1+\chi_{i=j})t_i^2t_j^2}}{2\sigma_r} \hspace{0.5cm} (1 \leq i \leq n)\]
(the indicator functions ensure \(\sum_{1 \leq i \leq n}{\mathbb{P}(X^{(r)}=i)}=1\)).
\par
Since \(M_{j_1j_2} \leq \frac{K}{2^{w}},\) it follows that
\begin{equation}\label{bdd0}
    \sum_{0 \leq h \leq w}{2^h \cdot (\tilde{d}_{2,M}(t,\mathcal{A}_{2,h}(M)))^2} \leq \sum_{0 \leq h \leq w}{2^h \cdot (\tilde{d}_{2,M}(t,(0,0,\hspace{0.05cm}...\hspace{0.05cm},0)))^2} \leq 2^{w+1} \cdot \frac{K}{2^{w}} \leq 2K
\end{equation}
via (\ref{maxd}). For \(h \geq w+1,\) let
\[(X_j^{(r)})_{1 \leq j \leq \lceil 2^{h-w-r} \cdot \sigma_r \rceil,0 \leq r \leq h-w,\sigma_r>0}\]
be independent random variables defined on the same sample space \(\Omega\) with \(X_j^{(r)}\overset{d}{=}X^{(r)},\) denote by \(X\) their concatenation into a vector: i.e., 
\[X=((X_j^{(0)})_{1 \leq j \leq \lceil 2^{h-w} \cdot \sigma_0 \rceil},(X_j^{(1)})_{1 \leq j \leq \lceil 2^{h-w-1} \cdot \sigma_1 \rceil},\hspace{0.05cm}...\hspace{0.05cm},(X_j^{(h-w)})_{1 \leq j \leq \lceil \sigma_{h-w} \rceil}),\]
the \(r^{th}\) block being empty when \(\sigma_r=0,\) and take
\[y_{r}=y_{h,r}=\frac{\lceil M(h) \sigma_{r} \rceil}{M(h)}+\frac{1}{2^{h-w-r}} \hspace{0.5cm} (0 \leq r \leq h-w).\]  
\par
Then \(y=y(h)=(y_{h,r})_{0 \leq r \leq h-w} \in \mathcal{S}_h\) because \(\lceil x \rceil \leq x+1\) gives
\[\sum_{0 \leq r \leq h-w}{y_r} \leq \sum_{0 \leq r \leq h-w}{\sigma_{r}}+\frac{h-w+1}{M(h)}+2 \leq 1+\frac{h}{2^{h-2}}+2 \leq 5,\]
and \(M(h) \cdot y_i \in \mathbb{Z}_{\geq 0}\) from \(M(h) \geq 2^{h-2}\) (proved below (\ref{f7})) entailing \(M(h) \geq 2^{h-w}\) as \(w \geq 2,\) or \(w=1,h \geq 2,2^h=(1+1)^h \geq 2\binom{h}{2}=h(h-1).\) Since \(V_r(X)=\{X_j^{(r)},1 \leq j \leq \lceil 2^{h-w-r} \cdot \sigma_r \rceil\}\) satisfy 
\[|V_r(X)| \leq \lceil 2^{h-w-r} \cdot \sigma_r \rceil \leq 2^{h-w-r} \cdot \sigma_r+1 \leq 2^{h-w-r} \cdot y_r \hspace{0.5cm} (0 \leq r \leq h-w),\]
it suffices to show that
\begin{equation}\label{eup}
    \mathbb{E}[(\tilde{d}_{2,M}(t,\hat{t}_{X,h}))^2] \leq \sum_{j:v_M(t,j) \leq \frac{K}{2^h}}{v_M(t,j) \cdot t_j^2}+\sum_{j}{v_M(t,j) \cdot t_j^2\cdot e^{-\frac{2^{h-2} \cdot v_M(t,j)}{K}}}+
\sum_{j}{v_M(t,j)\cdot (|t_j|-\sqrt{\frac{\lfloor 2^h \cdot t_j^2 \rfloor}{2^h}})^2}
\end{equation}
for the random variable \(\hat{t}_{X,h} \in \mathcal{A}_{2,h}(M)\) given by \(V(X)=\cup_{0 \leq r \leq h-w}{N_r(M,V_r(X))}\) and
\begin{equation}\label{deft}
    (\hat{t}_{X,h})_j=sgn(t_j) \cdot \sqrt{\frac{\lfloor 2^h \cdot t_j^2 \rfloor}{2^h}} \cdot \chi_{j \in V(X)} \hspace{0.5cm} (j \in \{1,2,\hspace{0.05cm}...\hspace{0.05cm},n\}).
\end{equation}
This inequality will imply that for each \(h \geq w+1,\) there exists \(\hat{t}_h=\hat{t}_{X,h}(\omega_h) \in \mathcal{A}_{2,h}(M),\omega_h \in \Omega\) with
\[\sum_{h \geq w+1}{2^h \cdot (\tilde{d}_{2,M}(t,\hat{t}_h))^2} \leq \sum_{(h,j):v_M(t,j) \leq \frac{K}{2^h}}{2^h \cdot v_M(t,j) \cdot t_j^2}+\sum_{j}{t_j^2\sum_{h \geq w+1}{2^{h}v_M(t,j) \cdot e^{-\frac{2^{h-2} \cdot v_M(t,j)}{K}}}}+\]
\begin{equation}\label{bdd}
    +\sum_{j}{\sum_{h \geq w+1}{2^h\cdot v_M(t,j)\cdot (|t_j|-\sqrt{\frac{\lfloor 2^h \cdot t_j^2 \rfloor}{2^h}})^2}} \leq 2K+4K(2+\frac{20}{7})\sum_{j}{t_j^2}+4\sum_{j}{v_M(t,j)}
\end{equation}
via a similar rationale as for (\ref{smallv}),
\begin{equation}\label{xx}
    \sum_{h \geq 0}{x \cdot 2^h \cdot e^{-x \cdot 2^h}} \leq 2+\frac{2}{e-2}<2+\frac{20}{7} \hspace{0.5cm} (x \in [0,1]),
\end{equation}
and (\ref{ineqxx}), respectively. Inequality (\ref{xx}) follows from
\[\sum_{h \geq 0}{x \cdot 2^h \cdot e^{-x \cdot 2^h}} \leq \sum_{h \geq 0, 2^h \leq x^{-1}}{x \cdot 2^h}+\sum_{h \geq 0, 2^h> x^{-1}}{x \cdot 2^h \cdot e^{-x \cdot 2^h}} \leq\]
\[\leq x \cdot 2x^{-1}+\frac{1}{1-\frac{2}{e}} \cdot x \cdot 2^{h_x+1} \cdot e^{-x \cdot 2^{h_x+1}} \leq 2+\frac{1}{1-\frac{2}{e}} \cdot 2 \cdot e^{-1}=2+\frac{2}{e-2}<2+\frac{20}{7}\]
by \(\frac{x \cdot 2^{h+1} \cdot e^{-x \cdot 2^{h+1}}}{x \cdot 2^h \cdot e^{-x \cdot 2^h}}=2 \cdot e^{-x \cdot 2^h} \leq \frac{2}{e}\) in the second sum when \(x>0,\) and 
\[h_x=\lfloor \log_2{(x^{-1})} \rfloor \in (\log_2{(x^{-1})}-1,\log_2{(x^{-1})}], \hspace{0.5cm} h_x \geq 0.\] 
The bounds (\ref{bdd0}) and (\ref{bdd}) entail 
\[\sum_{h \geq 0}{2^h \cdot (\tilde{d}_{2,M}(t,\mathcal{A}_{2,h}(M)))^2} \leq 2K+K \cdot [2+8+14+4]=30K,\]
whereby the desired inequality ensues from
\[\sum_{h \geq 0}{2^h \cdot (\tilde{d}_{2,M}(t,\tilde{\mathcal{A}}_{2,h}))^2} \leq 15K+\sum_{h \geq 5}{2^h \cdot (\tilde{d}_{2,M}(t,\mathcal{A}_{2,h-4}(M)))^2} \leq 15K+16 \cdot 30K=495K.\]
\par
It remains to justify (\ref{eup}): linearity of expectation yields this is a sum over \(j \in \{1,2,\hspace{0.05cm}...\hspace{0.05cm},n\}.\) In light of definition (\ref{deft}), the contribution of the entries \(j\) with \(v(t,j) \leq \frac{K}{2^{h}}\) is at most the first sum in the claimed bound. Consider next \(j\) with \(v(t,j)> \frac{K}{2^{h}}:\) their contribution to the expectation is shown to be at most
\[v_M(t,j) \cdot t_j^2\cdot e^{-\frac{2^{h-1}\sum_{0 \leq r \leq h-w}{v_r(t,j)}}{K}}+v_M(t,j)\cdot (|t_j|-\sqrt{\frac{\lfloor 2^h \cdot t_j^2 \rfloor}{2^h}})^2,\]
and this is enough since
\[\sum_{0 \leq r \leq h-w}{v_r(t,j)} \geq v_M(t,j)-\frac{K}{2^{h+1}} \geq \frac{v_M(t,j)}{2}.\]
\par
The above claim can be argued as follows: for \(0 \leq r \leq h-w,\sigma_r>0,\) and \(j \in \{1,2,\hspace{0.05cm}...\hspace{0.05cm},n\},\) 
\[\mathbb{P}(j \not \in \cup_{1 \leq y \leq \lceil 2^{h-w-r} \cdot \sigma_r \rceil}{N_r(M,X_j^{(r)}}))=(\mathbb{P}(j \not \in N_r(M,X^{(r)}))^{\lceil 2^{h-w-r} \cdot \sigma_r \rceil},\]
and
\[\mathbb{P}(j \not \in N_r(M,X^{(r)}))=1-\sum_{i \in N_r(M,j)}{\mathbb{P}(X^{(r)}=i)}=1-\sum_{i \in N_r(M,j)}{\frac{\sum_{ik \in E_r}{(1+\chi_{i=k})t_i^2t_k^2}}{2\sigma_r}} \leq 1-\sum_{i \in N_r(j)}{\frac{t_i^2t_j^2}{2\sigma_r}},\]
from which
\[\mathbb{P}(j \not \in \cup_{0 \leq r \leq h-w,\sigma_r>0}{(\cup_{1 \leq y \leq \lceil 2^{h-w-r} \cdot \sigma_r \rceil}{N_r(M,X_y^{(r)}}))})\leq \prod_{0 \leq r \leq h-w}{(1-\sum_{i \in N_r(M,j)}{\frac{t_i^2t_j^2}{2\sigma_r}})^{\lceil 2^{h-w-r} \cdot \sigma_r \rceil}} \leq\]
\[\leq e^{-\sum_{0 \leq r \leq h-w}{2^{h-w-r} \cdot \sigma_r \cdot \sum_{i \in N_r(M,j)}{\frac{t_i^2t_j^2}{2\sigma_r}}}}=e^{-\frac{2^{h-1}\sum_{0 \leq r \leq h-w}{v_r(t,j)}}{K}}\]
via \(1-x \leq e^{-x}.\) 
\end{proof}

\subsection{Two Lemmas}\label{almost}

This last subsection contains bounds on the sets the admissible sequences are derived from.

\begin{lemma}\label{sizeas}
    The sets in the sequences given by (\ref{a00h}) and (\ref{a11h}) satisfy
    \[|\mathcal{A}_{0,h}| \leq l_h^{5}, \hspace{0.5cm}|\mathcal{A}_{1,h}(B_{sym,r})| \leq l_h^3 \hspace{0.5cm} (1 \leq r \leq 6).\]
\end{lemma}

\begin{proof}
   Begin with \(\mathcal{A}_{0,h}:\) its definition yields
\[|\mathcal{A}_{0,h}| \leq |\mathcal{C}_h| \cdot \max_{z \in \mathcal{C}_h}{|\mathcal{D}_h(z)|} \leq (\sum_{0 \leq k \leq 2^h}{\binom{2^h-1+k}{2^h-1}}) \cdot (2^{2^{h}} \cdot \sum_{0 \leq k \leq 2^{h}}{\binom{2^{h}-1+k}{2^{h}-1}})\]
by using (\ref{key}) and \(|I_h(z)| \leq 2^{h}\) for \(z \in \mathcal{C}_h,\) a consequence of \(||z|| \leq 1\) and the natural analog of (\ref{keyv}) for \(z\) (the second factor above accounts for the signs \((s_j)_{2 \leq j \leq \min{(d,2^h)}}\)). Hence
\[|\mathcal{A}_{0,h}| \leq \binom{2^{h}+2^{h}}{2^h} \cdot 2^{2^{h}} \cdot \binom{2^{h}+2^{h}}{2^{h}} \leq 2^{2^{h}+2^{h}} \cdot 2^{2^{h}} \cdot 2^{2^{h}+2^{h}}=2^{5 \cdot 2^{h}}=l_h^5\]
by 
\[\binom{a-1+k}{a-1}=\binom{a+k}{a}-\binom{a+k-1}{a}, \hspace{0.5cm} \binom{n}{m} \leq 2^n.\]
\par
Continue with \(\mathcal{A}_{1,h}(B_{sym,r}):\) the claim ensues from (\ref{f7}) and (\ref{f8}) below alongside \(l_{h+1} \cdot l_h=l_h^3.\) Fix \(q,r:\) by definition,
\[|\mathcal{A}_{1,h}(B_{sym,r})| \leq |\mathcal{L}_h| \cdot \max_{y \in \mathcal{L}_h}{\prod_{1 \leq i \leq h_0}{|\mathcal{A}^{(h_0)}_{h+\lfloor \log_2{y_{i}} \rfloor}|}},\]
with
\begin{equation}\label{f7}
    |\mathcal{L}_h|=\sum_{0 \leq k \leq 2^{\lfloor 2^h/h \rfloor}}{\binom{k+h-1}{h-1}}=\binom{2^{\lfloor 2^h/h \rfloor}+h}{h} \leq (2^{\lfloor 2^h/h \rfloor}+h)^h \leq 2^{h \cdot (2^h/h+1)}=2^{2^h+h} \leq l_{h+1}
\end{equation}
from \(h \leq 2^h/h+1 \leq 2^{2^h/h}:\) induction gives \(2^h \geq h^2\) for \(h \geq 4,\) while for \(h \leq 3,\) \(2^h=2 \cdot 2^{h-1} \geq 2(1+h-1)=2h.\) Lastly,
\begin{equation}\label{f8}
    \max_{y \in \mathcal{L}_h}{\prod_{1 \leq i \leq h_0}{|\sqrt{y_i}\mathcal{A}^{(i)}_{h+\lfloor \log_2{y_{i}} \rfloor}|}} \leq \max_{y \in \mathcal{L}_h}{\prod_{1 \leq i \leq h_0}{l_{h}^{y_i}}} \leq l_h
\end{equation}
because when \(y \in \mathcal{L}_h,\) \(\sum_{1 \leq i \leq h_0}{y_i} \leq \sum_{1 \leq i \leq h}{y_i} \leq 1,\) and either \(y_i=0,|\sqrt{y_i}\mathcal{A}^{(i)}_{h+\lfloor \log_2{y_{i}} \rfloor}|=1,\) or \(\newline y_i>0,|\sqrt{y_i}\mathcal{A}^{(i)}_{h+\lfloor \log_2{y_{i}} \rfloor}| \leq l_{h+\lfloor \log_2{y_i} \rfloor} \leq l_{h+\log_2{y_i}}=l_h^{y_i}\) via
\[|\mathcal{A}^{(i)}_{k}| \leq l_k:\]
this last claim is clear for \(k \leq 4\) as \(l_k>1,\) while for \(k \geq 5,\) this is a consequence of Lemma~\ref{admlem} applied to \(w=i+1,K=16C^2,M=\overline{B}_{sym,r,i}\) (\(K=16C^2\) by virtue of \(|q| \leq 2\) and (\ref{fastdecay}): see also proof of Theorem~\ref{mainth2}).
\end{proof}

\begin{lemma}\label{admlem}
    Let \(w \in \mathbb{N},K>0\) be fixed. If \(M \in \mathbb{R}_{\geq 0}^{n\times n},n \leq l_w\) satisfies (\ref{stronger}), then
    the sets
    \[\tilde{\mathcal{A}}_{2,h}=\mathcal{A}_{2,\max{(h-4,0)}}(M) \hspace{0.5cm} (h \geq 0)\] 
    form an admissible sequence for \(\{x \in \mathbb{R}^{n}:||x|| \leq 1\}.\)
\end{lemma}

\begin{proof}
It is immediate that \(\mathcal{A}_{2,h}(M) \subset \{x \in \mathbb{R}^{n}:||x|| \leq 1\}\) for all \(h \geq 0.\) 
By definition,
\[|\mathcal{A}_{2,0}(M)|=1, \hspace{0.5cm}|\mathcal{A}_{2,h}(M)| \leq l_h \hspace{0.8cm} (1 \leq h \leq w),\]
and it suffices to justify that for \(h \geq w+1,\)
\begin{equation}\label{bda2h}
    |\mathcal{A}_{2,h}(M)| \leq  2^{2^h+3h} \cdot l^5_{h-1} \cdot l^{5}_{h} \cdot l_{h+1} 
\end{equation}
since 
\[\frac{l_{h+4}}{2^{2^h+3h} \cdot l^5_{h-1} \cdot l^{5}_{h} \cdot l_{h+1} }=\frac{l_{h-1}^{32-5-10-4}}{l_{h-1}^2 \cdot 2^{3h}}=2^{11 \cdot 2^{h-1}-3h} \geq 2^{8h}>1.\]
\par
The construction yields
\[|\mathcal{A}_{2,h}(M)| \leq |\mathcal{S}_h| \cdot \max_{y \in \mathcal{S}_h}{|\mathcal{P}_{h}(y)|} \cdot \max_{y \in \mathcal{S}_{h},(V_0,V_1,\hspace{0.05cm}...\hspace{0.05cm},V_{h-w}) \in \mathcal{P}_{h}(y)}{|\mathcal{D}_h(\cup_{0 \leq j \leq h-w}{N_j(V_j)})|},\]
whereby (\ref{bda2h}) is a consequence of (\ref{f1}),(\ref{f2}),(\ref{f3}) below.
Similarly to (\ref{f7}),
\begin{equation}\label{f1}\tag{\(f1\)}
    |\mathcal{S}_h|=\binom{5M(h)+h-w+1}{h-w+1} \leq (7M(h))^{h-w+1} \leq 2^{(2^h/h+3)(h-w+1)} \leq 2^{2^h+3h}
\end{equation}
since \(h \leq 2^h/h+1 \leq 2^{2^h/h} \leq 2M(h)\) (via the argument preceding (\ref{f8}), and \(2^{x} \geq x+1\) for \(x \geq 1\)). Next, \(n \leq l_w\) gives that for \(y \in \mathcal{S}_h,\)  
\[|\mathcal{P}_{h}(y)| \leq \prod_{0 \leq j \leq h-w}{\sum_{0 \leq k \leq 2^{h-w} \cdot y_j}{\binom{l_{w}}{k}}} \leq \prod_{0 \leq j \leq h-w}{\binom{l_{w}+\min{(\lfloor 2^{h-w} \cdot y_j \rfloor},l_{w})}{\min{(\lfloor 2^{h-w} \cdot y_j \rfloor},l_{w})}} \leq\]
\begin{equation}\label{f2}\tag{\(f2\)}
    \leq \prod_{0 \leq j \leq h-w}{(2l_{w})^{\lfloor 2^{h-w} \cdot y_j \rfloor}} \leq  (2l_{w})^{5 \cdot 2^{h-w}} \leq l^5_{h-1} \cdot l^{5}_{h}
\end{equation}
via 
\[\sum_{0 \leq k \leq a}{\binom{q}{k}} \leq \sum_{0 \leq k \leq a}{\binom{q}{k}\binom{a}{a-k}}=\binom{q+a}{a} \hspace{0.5cm} (a \in \mathbb{N}),\] 
and \(\binom{q+m}{m} \leq (q+m)^m \leq (2q)^m\) \((q \in \mathbb{N},m \in \mathbb{Z}_{\geq 0},q \geq m).\) Lastly,
\begin{equation}\label{f3}\tag{\(f3\)}
    \max_{y \in \mathcal{S}_{h},(V_0,V_1,\hspace{0.05cm}...\hspace{0.05cm},V_{h-w}) \in \mathcal{P}_{h}(y)}{|\mathcal{D}_h(\cup_{0 \leq j \leq h-w}{N_j(M,V_j)})|} \leq l_{h+1}:
\end{equation}
for such configurations, \(V=\cup_{0 \leq j \leq h-w}{N_j(M,V_j)} \subset \{1,2,\hspace{0.05cm}...\hspace{0.05cm},n\}\) has
\[|V| \leq \sum_{0 \leq j \leq h-w}{2^{w+j} \cdot |V_j|} \leq 2^h\sum_{0 \leq j \leq h-w}{y_j} \leq 2^h\]
via \(|N_s(M,v)| \leq 2^{w+s}\) for all \(s \geq 0,\) whereby the size of the set of interest is
\[\binom{2^h+|V|}{|V|}=\binom{2^h+|V|}{2^h} \leq \binom{2^h+2^h}{2^h} \leq 2^{2^h+2^h}=l_{h+1}\]
(by an analogous rationale as that for (\ref{f7})).
\end{proof}

This concludes the proof of Theorem~\ref{mainth}. The last section consists of a discussion on what goes wrong when trying to extend the arguments in sections~\ref{sect1} and \ref{sectp} to the series underpinning the \(\gamma_2\)-functional.

\section{Failures and Potential Extensions}\label{subsect6}

\par
Return now to the original series of interest, the one underlying \(\gamma_2,\) i.e., \(\sum_{h \geq 0}{2^{h/2} \cdot d_2(t,\hat{t}_{h})},\) to see what difficulties arise, even in the simple case (\ref{easy}) covered in section~\ref{sect1}. The computations therein 
yield
\[\sum_{h \geq 1}{2^{h/2} \cdot d_2(t,\hat{t}_{h})} \leq \sum_{1 \leq h \leq d_0+1}{\sqrt{\sum_{i \geq h}{2^{h-i}\sum_{j \in S_i}{t_j^4}}}}+\sum_{h \geq 1}{\sqrt{\sum_{i\leq h}{\sum_{j \in S_i}{2^{h-i} \cdot (t_{j}^2-\frac{\lfloor 2^{h-i} \cdot t^2_{j} \rfloor}{2^{h-i}})^2}}}}:=I_g+II_g.\]
\par
For the first term, 
\[I_g \leq \sum_{1 \leq h \leq d_0+1}{\sqrt{\sum_{i\geq h}{s_i^2 \cdot 2^{h-i}}}},\]
where
\[s_i=\sum_{j \in S_i}{t_j^2} \hspace{0.5cm} (1 \leq i \leq d_0+1).\]
Let
\[e_h=\chi_{h \leq d_0+1} \cdot \sqrt{\sum_{i \geq h}{s_i^2 \cdot 2^{h-i}}} \hspace{0.5cm} (h \in \mathbb{N}):\]
\(\sqrt{a+b} \leq \sqrt{a}+\sqrt{b}\) \((a,b \geq 0)\) gives
\[e_h-s_h \leq \sqrt{\sum_{i \geq h+1}{s_i^2 \cdot 2^{h-i}}}=2^{-1/2} \cdot e_{h+1},\]
from which
\[\sum_{h \in \mathbb{N}}{e_h} \leq \sum_{1 \leq i \leq d_0+1}{s_i}+2^{-1/2}\sum_{h \in \mathbb{N}}{e_h}.\]
Since \(\sum_{h \in \mathbb{N}}{e_h} \leq \sqrt{2}(d_0+1)<\infty,\) the following holds
\begin{equation}\label{big}
    I_g \leq \sum_{h \in \mathbb{N}}{e_h} \leq (1-2^{-1/2})^{-1}\sum_{1 \leq i \leq d_0+1}{s_i}=(1-2^{-1/2})^{-1}.
\end{equation}
\par
For the second term, split each term into two, based on the proof of (\ref{xineq7}):
\[II_g \leq \sum_{h \geq 1}{\sqrt{\sum_{(i,j,s): i \leq h, j \in S_i,t_j^2 \in (2^{-s},2^{-s+1}],1 \leq s \leq h-i}{\frac{1}{2^{h-i}}}}}+\sum_{h \geq 1}{\sqrt{\sum_{(i,j,s): i \leq h, j \in S_i,t_j^2 \in (2^{-s},2^{-s+1}],s \geq h-i+1}{2^{h-i} \cdot t_j^4}}}:=\]
\[:=\sum_{h \geq 1}{f_h}+\sum_{h \geq 1}{g_h},\]
that is, by using again \(\sqrt{a+b} \leq \sqrt{a}+\sqrt{b}\) alongside \(0 \leq x-\frac{\lfloor 2^{k} \cdot x \rfloor}{2^{k}} \leq \min{(2^{-k},x)}\) for \(a,b,x \geq 0.\)
\par
Consider \((f_h)_{h \geq 1}.\) For \(h \geq 2,\) in the conditions for the sum in \(f_h^2,\) \(i=h\) makes no contribution, and thus, \(i \leq h-1\) with 
\[f_h \leq 2^{-1/2}f_{h-1}+\sqrt{\sum_{(i,j,s): i \leq h, j \in S_i,t_j^2 \in (2^{-s},2^{-s+1}],s=h-i}{\frac{1}{2^{h-i}}}} \leq 2^{-1/2}f_{h-1}+\sqrt{\sum_{j \in S_i,t_j^2 \in (2^{-(h-i)},2^{-(h-i)+1}]}{t_j^2}},\]
which gives as for \((e_h)_{h \in \mathbb{N}}\) that
\[\sum_{h \geq 2}{f_h} \leq (1-2^{-1/2})^{-1} \cdot \sum_{h \geq 1}{\sqrt{\sum_{j \in S_i,t_j^2 \in (2^{-(h-i)},2^{-(h-i)+1}]}{t_j^2}}}\]
since \(f_1=0.\) Unfortunately, the last series (which is also a lower bound for \(2^{1/2} \cdot II_g\)) is not uniformly bounded in \(d_0:\) a unit vector with at least \(\frac{2^{k-2}}{k^2}\) entries equal to \(2^{-k/2}\) among those with positions in \(S_{d_0}\) for \(1 \leq k \leq 2^{d_0}\) has the corresponding sum at least 
\[\sum_{1 \leq k \leq 2^{d_0}}{\frac{1}{2k}}=\Theta(d_0).\]
\par
\vspace{0.5cm}
By virtue of the above failure, a more nuanced approximation of \(d(t,\mathcal{A}_h)\) is needed, and the treatment of the diagonal blocks from subsection~\ref{subsec23} provides a possible avenue for this, the probabilistic method. A considerable difficulty with the random variables used therein is that their corresponding proxies for \(d(t,\mathcal{A}_h)\) do not have straightforward expectations with the bound yielded by Cauchy-Schwarz suffering of a similar issue to that of \(II_g\) above. A more careful analysis might bypass this issue, but it is not yet clear how this could be accomplished.

\bibliography{references}

\end{document}